\documentclass[reqno,twoside]{amsart} 

\usepackage{amsmath,amsthm,amssymb,amsfonts,mathrsfs,color}
\usepackage{stmaryrd}
\usepackage{empheq}
\usepackage{cleveref}
\usepackage{dsfont}
\usepackage{yhmath}

\usepackage{enumitem}
\newcounter{saveenum}

\usepackage{pgfplots}
\pgfplotsset{compat=1.15}
\usepackage{mathrsfs}
\usetikzlibrary{arrows}
\usepackage{relsize}

\definecolor{qqttzz}{rgb}{0.,0.2,0.6}
\definecolor{uuuuuu}{rgb}{0.26666666666666666,0.26666666666666666,0.26666666666666666}
\definecolor{ffqqqq}{rgb}{1.,0.,0.}
\definecolor{qqqqff}{rgb}{0.,0.,1.}

\theoremstyle{plain}
\begingroup
\newtheorem{thm}{Theorem}[section]

\newtheorem{lem}[thm]{Lemma}
\newtheorem{prop}[thm]{Proposition}

\endgroup

\theoremstyle{definition}
\begingroup

\newtheorem{rem}[thm]{Remark}
\endgroup

\numberwithin{equation}{section}

\newcommand{\res}{\mathop{\hbox{\vrule height 7pt width .5pt depth 0pt
\vrule height .5pt width 6pt depth 0pt}}\nolimits}

\newcommand{\ds}{\displaystyle}

\newcommand{\N}{\mathbb N} 
\newcommand{\R}{\mathbb R}

\newcommand{\Ms}{{\mathbb M}^{N{\times}N}_{\rm sym}}

\newcommand{\E}{{\mathcal E}}
\newcommand{\F}{{\mathcal F}}
\newcommand{\G}{{\mathcal G}}

\newcommand{\e}{\varepsilon}

\newcommand{\LL}{{\mathcal L}}
\newcommand{\HH}{{\mathcal H}}
\newcommand{\KK}{ K}
\newcommand{\M}{{\mathcal M}}
\newcommand{\C}{{\mathcal C}}
\newcommand{\A}{{\mathcal A}}
\newcommand{\T}{{\mathcal T}}

\newcommand{\infconv}{ \, {\scriptstyle \Box}  \, }

\renewcommand{\bar}[1]{\overline{#1}}
\newcommand{\norme}[1]{ \left\Vert #1 \right\Vert}

\let\O=\Omega

\begin{document}
 
\title[Perfect plasticity versus Damage]{Perfect plasticity versus Damage: an unstable interaction between irreversibility and $\Gamma$-convergence through variational evolutions.}
\author[E. Bonhomme]{\'Elise Bonhomme}

\address[E. Bonhomme]{Universit\'e Paris-Saclay, CNRS,  Laboratoire de math\'ematiques d'Orsay, 91405, Orsay, France.}
\email{elise.bonhomme@universite-paris-saclay.fr}

\begin{abstract} 
This paper addresses the question of the interplay between relaxation and irreversibility through quasi-static evolutions in damage mechanics, by inquiring the following question: can the quasi-static evolution of an elastic material undergoing a rate-independent process of plastic deformation be derived as the limit model of a sequence of quasi-static brittle damage evolutions? This question is motivated by the static analysis performed in \cite{BIR}, where the authors have shown how the brittle damage model introduced by Francfort and Marigo (see \cite{FM1,FM}) can lead to a model of Hencky perfect plasticity. Problems of damage mechanics being rather described through evolution processes, it is natural to extend this analysis to quasi-static evolutions, where the inertia is neglected. We consider the case where the medium is subjected to time-dependent boundary conditions, in the one-dimensional setting. The idea is to combine the scaling law considered in \cite{BIR} with the quasi-static brittle damage evolution introduced in \cite{FG} by Francfort and Garroni, and try to understand how the irreversibility of the damage process will be expressed in the limit evolution. Surprisingly, the interplay between relaxation and irreversibility is not stable through time evolutions. Indeed, depending on the choice of the prescribed Dirichlet boundary condition, the effective quasi-static damage evolution obtained may not be of perfect plasticity type.
\end{abstract}

\maketitle



\section{Introduction}

\subsection{Interplay between $\Gamma$-convergence and variational evolutions}
Rate-independent systems have proved to be very useful in many problems of continuum mechanics dealing with dissipative phenomena such as elastoplasticity, damage or fracture. These models share similar energetic formulations which put in competition a stored energy and a dissipated one which does not depend on the speed of the loading (see \cite{MM} and references therein). When the model involves a scaling parameter, say $\e >0$, a natural attempt consists first in studying static models before treating the related evolutions. Such considerations can be dealt within a more general setting, such as in \cite{MRS} where the authors derive a sufficient condition in order for a family of parametrized time-dependent energy functionals of the form $E_\e + D_\e$ to approximate the expected effective energy $E_0 + D_0$ in the limit, where $E_\e$ and $D_\e$ respectively stand for the stored and dissipated parts of the total energy and $E_0$ and $D_0$ are their corresponding $\Gamma$-limits. In a nutshell, for $\e \in [0,\infty)$, if $E_\e$ stands for the stored energy and $D_\e(q,\tilde q)$ stands for the minimal energy dissipated as the medium changes from the states $q$ to $\tilde q$, an evolution $q_\e$ is called an energetic solution during the time interval $[0,T]$ if it satisfies the following stability and energy balance conditions
$$
E_\e(t,q_\e(t)) \leq E_\e(t,q) + D_\e (q_\e(t),q) \quad \text{for all admissible state } q
$$
$$
E_\e(t,q_\e(t)) + {\rm Diss}_\e(q_\e;0,t) = E_\e(0,q_\e(0)) + \int_0^t \partial_s E(s,q_\e(s)) \, ds
$$
for all time $t \in [0,T]$, where the cumulated dissipation ${\rm Diss}_\e(q_\e;0,t)$ is the total variation of $q_\e$ with respect to the "distance" $D_\e$ in the time interval $[0,t]$. Given energetic solutions $\{ q_\e \}_{\e >0}$ converging to some $q_0$, the authors derive in \cite[Theorem 3.1]{MRS} a sufficient condition in order for $q_0$ to be an energetic solution of the limit problem associated to $E_0$ and $D_0$. In particular, a joint condition on the interplay between the stored and dissipated energies is needed. Unfortunately, even if we have separate $\Gamma$-convergence of $E_\e$ and $D_\e$ to $E_0$ and $D_0$ respectively, the $\Gamma$-limit of the total energies $E_\e + D_\e$ might differ from the sum of the $\Gamma$-limits. This particular issue is adressed in \cite{BCGS} where the authors consider a family of quasi-static evolutions involving internal oscillating energies $E_\e$ and dissipations $D_\e$ and show that the $\Gamma$-limit of the sum can still be additively split as the sum of a stored energy $\tilde E_0$ and a dissipated one $\tilde D_0$, even though they a priori differ from $E_0$ and $D_0$. More generally, the interaction between $\Gamma$-convergence and variational evolutions frequently involves unexpected and tedious non-commutability phenomena in various contexts. For instance, such considerations have attracted renewed interest in the derivation of lower dimensional models for thin structures in the evolutionary setting, in the context of elastoplasticity \cite{DM,LM}, crack propagation \cite{Bfracture} or delamination problems \cite{MRT}, without being exhaustive. Another case study concerns the stability of unilateral minimality properties through variational evolutions, as in fracture mechanics \cite{GP} or periodic homogenization in multi-phase elastoplasticity \cite{FGiacomini}.

\subsection{Motivation and results}
In the static analysis led in \cite[Theorem 3.1]{BIR}, the authors consider a family of brittle damage energies (introduced in \cite{FM1,FM}) within a specific scaling law, and show how an asymptotic analysis in a singular limit can lead to a model of Hencky perfect plasticity. More precisely, they introduce a small parameter $\e >0$ and consider a linearly elastic material which can only exist in one of two states: a damaged one whose elastic properties are described via a symmetric fourth-order Hooke Law $ \e A_0$ and a sound one with a stronger elasticity tensor $A_1$, satisfying $ \e A_0 < A_1 $ in the sense of quadratic forms acting on $\Ms$. Introducing the characteristic function of the damaged region, $\chi \in L^\infty(\O;\{0,1 \})$, and following the model introduced by Francfort and Marigo, the total energy associated to a displacement $u \in H^1(\O;\R^N)$ and $\chi$ is given as the sum of the elastic energy stored inside the material and a dissipative cost, taken as proportional to the volume of the damaged zone:
$$
\int_\O \frac12 \left( \chi \e A_0 + (1-\chi) A_1 \right) e(u):e(u) \, dx + \frac{\kappa}{\e} \int_\O \chi \, dx 
$$
where $\kappa / \e  > 0$ is the material toughness and the symmetric gradient $e(u) =  \left( \nabla u + \nabla u ^T \right) / 2$ is the linearized elastic strain. As the parameter $\e$ tends to $0$, the elasticity coefficients of the weak material degenerate to zero while the diverging character of $\kappa / \e$ forces the damaged region to concentrate on vanishingly small sets. It is by now well-known that for fixed $\e > 0$, the minimization of the above energy with respect to the couple $(u,\chi)$ is ill-posed, so that the energy must be relaxed. By doing so, the brittle character of the damage is lost as minimizing sequences tend to develop microstructures and  the class of admissible solutions is extended to the set of all possible homogenized elasticities, resulting from fine mixtures of strong and weak material (see \cite{FM,FG,Allaire,AL}). Given some displacement $u$ and minimizing first pointwise with respect to $\chi$, one can check that the asymptotic analysis of these energies is equivalent to finding the $\Gamma$-limit of the family of functionals
$$
\int_\O \underset{=: W_\e(e(u))}{\underbrace{\min \left( \frac12 \e A_0 e(u):e(u) + \frac{\kappa}{ \e} \, ; \, \frac12 A_1 e(u):e(u) \right)}} \, dx 
$$
when $\e \searrow 0$, or still the $\Gamma$-limit of their lower semicontinuous envelopes, given by 
$$\E_\e(u) := \int_\O SQ W_\e(e(u)) \, dx $$ 
where $SQW_\e$ is the symmetric quasiconvex envelope of $W_\e$ (see \cite{AL}). An explicit formula of $SQW_\e$ is generally unknown, as its expression is obtained through a minimization among all attainable composite materials (the $G$-closure set, see \cite{Allaire}) and makes use of the Hashin-Shtrikman bounds. Slightly adapting the proof of \cite[Theorem 3.1]{BIR} (see Appendix \ref{section:appendix} for a precise statement and its proof), the authors have shown that when $A_0$ and $A_1$ are isotropic Hooke Laws defined by 
$$
A_i \xi = \lambda_i {\rm tr}(\xi) {\rm Id} + 2 \mu_i \xi \text{ for all } \xi \in \Ms
$$
where $\lambda_1 > \lambda_0 > 0$ and $\mu_1 > \mu_0 > 0$ are the Lam\'e coefficients, the brittle damage energies $ \E_\e$ $\Gamma$-converges in $L^1(\O;\R^N)$ as $\e \searrow 0$ to the functional
$$
\E : u \in BD(\O) \mapsto \int_\O \bar W(e(u)) \, dx + \int_\O I_\KK^* \left( \frac{ d E^su}{d|E^su|} \right) \, d |E^su| + \int_{\partial \O} I_\KK^* \left( (w-u) \odot \nu \right) \, d \HH^{N-1} 
$$
where $\KK = \{ \tau \in \Ms \, : \, G(\tau) \leq 2 \kappa \}$ is a closed convex set, $G : \Ms \to \R$ is defined by 
\begin{equation}\label{eq:G}
G(\tau) := \begin{cases}
\frac{\tau_1^2}{\lambda_0 + 2 \mu_0} & \text{ if } \frac{\lambda_0 + 2 \mu_0}{2(\lambda_0 + \mu_0)} \left( \tau_1 + \tau_N \right) < \tau_1, \\
\frac{(\tau_1 - \tau_N)^2}{4\mu_0} + \frac{ (\tau_1 + \tau_N)^2}{4(\lambda_0 + \mu_0)} & \text{ if } \tau_1 \leq  \frac{\lambda_0 + 2 \mu_0}{2(\lambda_0 + \mu_0)} \left( \tau_1 + \tau_N \right) \leq \tau_N , \\
\frac{\tau_N^2}{\lambda_0 + 2 \mu_0} & \text{ if } \tau_N <  \frac{\lambda_0 + 2 \mu_0}{2(\lambda_0 + \mu_0)} \left( \tau_1 + \tau_N \right) ,
\end{cases}
\end{equation}
with $\tau_1 \leq ... \leq \tau_N$ the ordered eigenvalues of $\tau \in \Ms$, $\overline W$ is the infimal convolution 
$$\overline W  : \xi \in \Ms \mapsto \underset{ \tau \in \Ms}{\inf} \, \left\{ \frac12 A_1 \tau:\tau + I_\KK^*( \xi - \tau) \right\}  $$
and 
$$I_\KK^* : \xi \in \Ms \mapsto \underset{ \tau \in \KK}{\sup} \, \left\{ \tau : \xi  \right\}$$
is the support function of $\KK$, standing for the plastic dissipation potential (see \cite{Suquet}). In particular, for all displacement $u \in BD(\O)$, writing the Radon-Nikod\'ym decomposition of $Eu$ with respect to Lebesgue
$$
Eu = e(u) \LL^N \res \O + E^s u
$$
and using the definition of the infimal convolution, we infer that the absolutely continuous linearized strain can be additively split as $e(u) = e + p^a$ with $e$ and $p^a \in L^1(\O;\Ms)$ such that 
$$W(e(u)) = \frac12 A_1 e:e + I_\KK^*(p^a) \quad \text{$\LL^1$-a.e. in $\O$.}$$
Therefore, defining $ p = E^s u + p^a \LL^N \res \O + (w - u) \odot \nu \HH^{N-1} \res \partial \O$, one can check that 
$$ 
Eu = e \LL^N \res \O + p \res \O,
$$
so that
$$
\E(u) = \int_\O \frac12 A_1 e:e \, dx + \int_{\bar \O} I_\KK^* \left( \frac{d p}{d \left\lvert p \right\rvert} \right) \, d \left\lvert p \right\rvert 
$$
is indeed the energy functional corresponding to Hencky perfect plasticity, as mentionned in \cite{Mora}.

\medskip

The objective of the present paper is to extend this work to the quasi-static case in a one-dimensional setting. More specifically, we consider a linearly elastic material whose reference configuration is $\O = (0,L)$, a bounded open interval, with toughness $\kappa > 0$ and stiffness $a_1 >0$, subjected to a prescribed time-dependent displacement on $\partial \O = \{ 0 , L \}$:
$$
w \in AC([0,T];H^1(\R)).
$$
Adapting the analysis led in \cite{BIR} to the quasi-static setting, we consider a family of quasi-static brittle damage evolutions (introduced in \cite{FG}) within the same specific scaling law. More precisely, we introduce a small parameter $\e >0$ and apply \cite[Theorem 2]{FG} to a linearly elastic material which can only exist in a damaged state or in a sound state with respective stiffness $ 0 < \e a_0 < a_1 $, subjected to the prescribed displacement $w$ on $\partial \O$ and with toughness $\kappa / \e$. Thus, we recover a triple
\begin{equation}\label{eq:brittle evolution}
(u_\e, \Theta_\e, a_\e) : [0,T] \to H^1(\O;\R) \times L^\infty(\O;[0,1]) \times L^\infty(\O;[0,a_1])
\end{equation} 
discribing the quasi-static evolution of brittle damage undergone by the medium for a fixed $\e >0$. In other words, the state of the damaged medium (for $\e >0$ fixed) at time $t \in [0,T]$ is dictated by the displacement $u_\e(t)$ while its elastic properties are given by the stiffness $a_\e(t) \in \G_{\Theta_\e(t)} (\e a_0,a_1)$ (see Section \ref{section:notation}), where $\Theta_\e(t)$ is the volume fraction of sound material $a_1$ (see Proposition \ref{prop:FG} below). We next wish to perform the asymptotic analysis of these evolutions when taking the limit $\e \searrow 0$, in the hope of recovering a quasi-static evolution of perfect plasticity in the limit, of which we briefly recall the fundamentals now.

\medskip

In \cite{Suquet}, Suquet proposed the (first complete) mathematical kinematical framework adapted to evolutions of perfect plasticity for dissipative materials and proves the existence of solutions in terms of the displacement field, under the assumption of small deformations. Heuristically, let $\O=(0,L)$ be the configuration at rest of an elastoplastic medium with stiffness $a_1$, whose evolution is driven by a time-dependent boundary displacement $w : [0,T] \times \R \to  \R$ prescribed on $\partial \O$. The behaviour of the material is described via three kinematic variables $(u,e,p)$, where the displacement $u : [0,T] \times \O  \to \R$ is such that the linearized strain $D u = e + p$ is additively decomposed in an elastic strain $e : [0,T] \times \O \to \R$ and a plastic strain $p : [0,T] \times \O \to \R$ accounting for the reversible and permanent deformations respectively. In the quasi-static setting, where inertia is neglected, the evolution satisfies the {\it Constitutive Equations}
$$\begin{cases}
			a_1 e(t) =  \sigma(t) \\
			\sigma(t) \in \KK  \\
			\dot p (t) : \sigma(t) = \underset{\tau \in \KK}{\sup} \, \dot p(t) : \tau  
		 \end{cases} \quad \text{in } \O$$
and the {\it Equilibrium Equation} 
$$\begin{cases}
 \sigma'(t) = 0 \text{ in }  \O  \\
u(t) = w(t) \text{ on } \partial \O 
\end{cases} $$
at all time $ t \in [0,T]$. In other words, the {\it Constitutive Equations} mean that the elastic strain is proportional to the stress $\sigma$, which is constrained to lie in a given closed and convex set $\KK \subset \R$ standing for the elasticity domain and whose boundary $\partial \KK$ is referred to as the yield surface. The last assertion is nothing but Hill's maximum work principle. More recently, quasi-static plastic evolutions have been revisited into a variational evolution formulation for rate-independent processes. The problem has been interpreted in an energetic form that does not require the solutions to be smooth in time nor in space, making use of modern tools of the calculus of variations instead (see \cite{MM,DMDSM,Mora} and references therein). Following \cite[Definition 4.2]{DMDSM}, a quasi-static evolution of perfect plasticity is a triple 
$$(u,e,p) : [0,T] \to BV(\O) \times L^2(\O;\R) \times \M (\bar \O; \R)$$
subjected to the relaxed boundary condition $p(t) \res \partial \O = (w(t) - u(t)) \left( \delta_L - \delta_0 \right) $ and satisfying the additive decomposition $Du(t) = e(t) \LL^1 \res \O + p(t) \res \O$, such that
\begin{equation}\label{eq:def plasticity}
\begin{gathered}
\sigma(t) = a_1 e(t), \quad \sigma'(t) = 0 \text{ in } H^{-1}(\O), \quad \sigma(t) \in \KK \, \LL^1 \text{-a.e. in } \O,   \\
p : [0,T] \to \M(\bar \O ; \R) \text{ has bounded variation} 
\end{gathered}
\end{equation}
and the Energy Balance
\begin{equation}\label{eq:EB}
\frac12 \int_\O a_1 e(t):e(t) \, dx + {\rm Diss}_\KK (p;0,t) = \frac12 \int_\O a_1 e(0):e(0) \, dx + \int_0^t \int_\O \sigma(s) \dot w'(s,x) \, dx ds
\end{equation}
holds, for all time $t \in [0,T]$. The dissipative plastic cost cumulated during the time interval $[0,t]$ associated to $p$, ${\rm Diss}_\KK(p;0,t)$, is defined as
$$
\sup \Biggl\{ \sum_{i=1}^{n} \int_{\bar \O} I_\KK^* \left( \frac{d (p(s_i) - p(s_{i-1}))}{d  \left\lvert p(s_i) - p(s_{i-1})\right\rvert } \right)  \, d \left\lvert p(s_i) - p(s_{i-1})\right\rvert \, : \, n \in \N, \, 0= s_0 \leq s_1 \leq ... \leq s_n = t  \Biggr\}	.		 
$$

\medskip

The existence of quasi-static evolutions is (by now classically) obtained by performing a time-discretization and solving incremental minimization problems inductively, before letting the time-step tend to $0$ (see \cite{MM, Cr1, Cr2, GLuss, DMDSM} for instance). The purpose of the present paper is not to prove the existence of quasi-static evolutions of perfect plasticity, but to establish whether such evolutions can be derived from the quasi-static brittle damage evolutions introduced above in \eqref{eq:brittle evolution}. By analogy with the static analysis of \cite{BIR}, we expect to derive the same closed convex set of plasticity $\KK$ which is given by the closed interval
$$
\KK := \big[ -\sqrt{2 \kappa a_0}, \sqrt{2 \kappa a_0} \, \big]
$$
as one can check that $G(\tau) = \tau^2 / a_0 $ for all $\tau \in \R$ in this simplified setting. In particular, the support function of $\KK$ is simply given by 
$$
I^*_\KK = \sqrt{2 \kappa a_0} \left\vert \, \cdot \, \right\rvert.
$$
Therefore, following the variational framework of quasi-static perfect plasticity recalled above (see \cite{DMDSM,MM}), the dissipative cost cumulated during a time interval $[s,t] \subset [0,T]$ due to a time dependent Radon measure $q :[0,T] \to \M([0,L])$ is defined as
$$ {\rm Diss}_\KK (q; s,t) = \sqrt{2 \kappa a_0} \, \mathcal{V}(q;s,t) $$
where
$$\mathcal{V}(q;s,t) = \sup \left\{  \sum_{i=1}^{n}   \left\lvert q(s_i) - q(s_{i-1})\right\rvert \big( [0,L] \big) \, : \, n \in \N, \, s = s_0 \leq s_1 \leq ... \leq s_n = t  \right\}$$ 
is the total variation of $q$ during the time interval $[s,t]$. The question inquired in the present work is then: when passing in the limit $\e \searrow 0$ (in some sense detailed in the next sections) in the above brittle damage evolutions \eqref{eq:brittle evolution}, can we derive a quasi-static evolution of perfect plasticity 
$$(u,e,p) : [0,T] \to BV((0,L)) \times L^2((0,L);\R) \times \M ([0,L]; \R) $$
satisfying \eqref{eq:def plasticity} and \eqref{eq:EB}? Contrary to the static analysis, the interplay between damage and $\Gamma$-convergence turns out to be unstable through the time evolution process. Indeed, as explained in Theorem \ref{thm:DAMAGE} and Theorem \ref{thm:CNS FR}, the effective quasi-static evolution derived in the subsequent sections might not be of perfect plasticity type. Instead, it can be interpreted as one of damage, characterised by means of the material's compliance as internal variable:
\begin{thm}\label{thm:DAMAGE}
Let $\e > 0$ and $(u_\e, \Theta_\e, a_\e)$ be a quasi-static evolution of the homogenized brittle damage model given by Proposition \ref{prop:FG}. There exists a subsequence (not relabeled) and absolutely continuous functions
$$(u,e,p,\sigma, \mu) : [0,T] \to BV((0,L)) \times \R \times \M([0,L]) \times \KK \times \M([0,L])$$
such that for all $t \in [0,T]$
\begin{subequations}
	\begin{empheq}[left=\empheqlbrace]{align}
	\, & u_\e(t) \rightharpoonup u(t)  \quad \text{weakly-* in } BV((0,L)), \nonumber \\
	\, & \mu_\e(t) := \frac{1 - \Theta_\e(t)}{\e} \LL^1 \res (0,L) \rightharpoonup \mu(t)  \quad \text{weakly-* in } \M([0,L]), \nonumber\\
	\, & \sigma_\e(t) \to \sigma(t)  \quad \text{in } \R, \nonumber \\
	\, & e_\e(t) := \frac{\sigma_\e(t)}{a_1} \Theta_\e(t) \rightharpoonup e(t)  \quad \text{weakly in } L^2((0,L)), \nonumber \\
	\, & p_\e(t) := \frac{\sigma_\e(t)}{a_0} \mu_\e(t) \rightharpoonup p(t) \quad \text{weakly-* in } \M([0,L]), \nonumber
	\end{empheq}
\end{subequations}
when $\e \searrow 0$ and satisfying the following assertions for all $t \in [0,T]$:
\begin{enumerate}[label=\roman*., leftmargin=* ,parsep=0.1cm,topsep=0.2cm]
\item Additive Decomposition: \quad $Du(t) = e(t)\LL^1 \res (0,L) + p(t) \res (0,L)$ in $\M((0,L))$
\item Relaxed Dirichlet Condition: \quad $p(t) \res \{ 0,L \} = \big(w(t) - u(t) \big) \big( \delta_L - \delta_0 \big)$ in $\M(\{0,L \})$ 
\item Constitutive Equation: \quad $\sigma(t) = a_1 e(t) $								
\item Equilibrium Equation: \quad $\sigma'(t) = 0$ in $H^{-1}((0,L))$
\item Stress Constraint: \quad $\sigma(t) \in \KK$.
\setcounter{saveenum}{\value{enumi}}
\end{enumerate} 
Furthermore, the effective compliance defined by
$$
c : t \in [0,T] \mapsto \frac{\mu(t)}{a_0} + \frac{1}{a_1} \LL^1 \res (0,L) \in \M([0,L];\R^+)
$$
is non-decreasing in time and satisfies the following assertions:
\begin{enumerate}[label=\roman*., leftmargin=* ,parsep=0.1cm,topsep=0.2cm]
\setcounter{enumi}{\value{saveenum}}
\item Constitutive Equation:\quad $Du(t) = \sigma(t) c(t) \res (0,L)  \text{ in } \M((0,L)) $ for all $t \in [0,T]$
\item Griffith Evolution Law: \quad $\dot c (t) \left( 2 \kappa a_0 - \sigma(t)^2 \right) = 0   \text{ in } \M([0,L];\R^+) $ for $ \LL^1 $-a.e. $ t \in [0,T]$.
\end{enumerate} 
\end{thm} 
\noindent
As mentionned above, according to the choice of the Dirichlet condition, the medium's response to the loading might differ from a perfect plastic behaviour:
\begin{thm}\label{thm:CNS FR}
The quasi-static evolution $(u,e,p) : [0,T] \to BV((0,L)) \times \R \times \M([0,L])$ is a quasi-static evolution of the perfect plasticity model \eqref{eq:def plasticity} and \eqref{eq:EB} if and only if the Dirichlet boundary condition $w \in AC([0,T];H^1(\R))$ satisfies: 
\begin{equation}\label{eq:CNS w}
\text{For all } 0 \leq s < t \leq T, \quad 
\left\lvert \big[ w(t) \big]^L_0 \right\rvert < \left\lvert \big[ w(s) \big]^L_0 \right\rvert \Rightarrow \left\lvert \big[ w(t) \big]^L_0 \right\rvert \leq \sqrt{2\kappa a_0} \frac{L}{a_1}.
\end{equation}
\end{thm}
\noindent
The present study may be seen as an illustration of non-stability issues arising when dealing with problems of $H$-convergence in the $L^1(\O)$ framework where ellipticity is lost during the time process, even when working in the simplest evolution setting and in dimension one.

\subsection{Organization of the paper} In Section \ref{section:notation}, we recall some notation and preliminary results. In Section \ref{section:FG}, we introduce the family of quasi-static brittle damage evolutions derived in \cite[Theorem 2]{FG} associated to a linearly elastic material with toughness $\kappa /\e$ and stiffness tensors $\e a_0$ and $a_1$, for $\e >0$, given a prescribed boundary datum and no volume force load in the one-dimensional setting. Particularly, due to the explicit knowledge of the $\G$-closure set of all admissible homogenized composite materials in dimension one, we collect starting information of crucial interest for the subsequent sections. In Section \ref{section:effective evolution}, we derive the effective quasi-static evolution when passing to the limit $\e \searrow 0$. We first give uniform bounds in Proposition \ref{prop:Bounded Energy} and next analyse the behaviour and regularity properties of the effective evolution. Section \ref{section:Energy Balance} adresses the question of the nature of the quasi-static evolution and determines in Theorem \ref{thm:CNS FR} the necessary and sufficient condition ensuring the perfect plastic behaviour of the evolution. Finally, Section \ref{section:ccl} discusses whether the present work could be improved in order to derive a quasi-static evolution of perfect plasticity.

\section{Notation and preliminaries} \label{section:notation}

\noindent \textbf{Matrices.} If $a$ and $b \in \R^N$, with $N \in \N \setminus \{ 0 \}$, we write $a \cdot b = \sum_{i=1}^N a_i b_i$ for the Euclidean scalar product and $\left\lvert a \right\rvert = \sqrt{a \cdot a}$ for the corresponding norm. The space of symmetric $N \times N$ matrices is denoted by $\Ms$ and is endowed with the Frobenius scalar product $\xi : \eta = {\rm tr} (\xi \eta)$ and the corresponding norm $\left\lvert \xi \right\rvert = \sqrt{\xi : \xi}$. 

\medskip

\noindent \textbf{Measures.} The Lebesgue and $k$-dimensional Hausdorff measures in $\R^N$ are respectively denoted by $\LL^N$ and $\HH^k$. If $X$ is a borel subset of $\R^N$ and $Y$ is an Euclidean space, we denote by $\M(X;Y)$ the space of $Y$-valued bounded Radon measures in $X$ which, according to the Riesz Representation Theorem, can be identified to the dual of $C_0(X;Y)$ (the closure of $C_c(X;Y)$ for the sup-norm in $X$). The weak-* topology of $\M(X;Y)$ is defined using this duality. The indication of the space $Y$ is omitted when $Y = \R$. For $\mu \in \M(X;Y)$, its total variation is denoted by $|\mu|$ and we denote by $\mu = \mu^a + \mu^s$ the Radon-Nikod\'ym decomposition of $\mu $ with respect to Lebesgue, where $\mu^a$ is absolutely continuous and $\mu^s$ is singular with respect to the Lebesgue measure $\LL^N$.

\medskip

\noindent \textbf{Functional spaces.}  We use standard notation for Lebesgue and Sobolev spaces. 
If $U$ is a bounded open subset of $\R^N$, we denote by $L^0(U;\R^m)$ the set of all $\LL^N$-measurable functions from $U$ to $\R^m$. We recall some properties regarding functions with values in a Banach space and refer to \cite{FL,Br,DMDSM} for details and proofs on this matter. If $Y$ is a Banach space and $T > 0$, we denote by $AC([0,T];Y)$ the space of absolutely continuous functions $f : [0,T] \to Y$. If $Y$ is the dual of a separable Banach space $X$, then every function $f \in AC([0,T];Y)$ is such that the weak-* limit 
$$\dot f(t) = \text{{\it w*-}}\lim_{s \to t} \frac{f(t) - f(s)}{t-s} \in Y$$
exists for $\LL^1$-a.e. $t \in [0,T]$, $\dot f :[0,T] \to Y$ is weakly-* measurable and $t \mapsto \Vert \dot f(t) \Vert_Y \in L^1([0,T])$. If $f : [0,T] \times \R \to \R$ is a function of two variables, time and spatial derivatives will be respectively denoted by $\dot f$ and $f'$.

\medskip

\noindent\textbf{Functions of bounded variation.} Let $U \subset \R^N$ be an open bounded set. A function $u \in L^1(U;\R^m)$ is a {\it function of bounded variation} in $U$, and we write $u \in BV(U;\R^m)$, if its distributional derivative $Du$ belongs to $\M(U;\mathbb M^{m \times N})$. We use standard notation for that space and refer to \cite{AFP} for details. We just recall that if $U$ has Lipschitz boundary, every function $u \in BV(U;\R^m)$ has an inner trace on $\partial U$ (still denoted by $u$ and $\HH^{N-1}$-integrable on $\partial U$) and there exists a constant $C > 0$ depending only on $U$ such that 
\begin{equation}\label{eq:norme BV}
\frac1C \norme{u}_{BV(U;\R^m)} \leq \lvert Du \rvert (U) + \int_{\partial U} \left\lvert u \right\rvert \, d \HH^{N-1} \leq C \norme{u}_{BV(U;\R^m)}
\end{equation}
according to \cite[Proposition 2.4, Remark 2.5 (ii)]{Temam}.

\medskip

\noindent\textbf{Homogenization and $H$-convergence} We refer to \cite{Allaire} for an exhaustive presentation of these notions and only recall minimal results. We denote, for fixed $\alpha,\beta >0$, the subset of fourth-order symmetric tensors
$$
\F_{\alpha,\beta} = \left\{ A \in \R^{N^4} \, : \, A_{ijkl}=A_{klij}=A_{jikl}, \, \alpha \left\lvert \xi \right\rvert^2 \leq A \xi:\xi \leq \beta \left\lvert \xi \right\rvert^2 \text{ for all } \xi \in \Ms \right\}.
$$
Let $\O$ be a bounded open set of $\R^N$. We say that $A_n \in L^\infty(\O;\F_{\alpha,\beta} )$ $H$-converges to $A \in L^\infty (\O;\F_{\alpha,\beta})$ if, for every $f \in H^{-1}(\O;\R^N)$, the solutions $u_n \in H^1_0(\O;\R^N)$ of the equilibrium equations 
$$\begin{cases}
 - {\rm div} \big( A_n e(u_n) \big) = f \quad \text{in } \O \\
 u_n = 0 \quad \text{on } \partial \O
 \end{cases}
$$
are such that $u_n \rightharpoonup u $ weakly in $ H^1_0(\O;\R^N) $ and $A_n e(u_n) \rightharpoonup A e(u) $ weakly in $ L^2(\O;\Ms)$ as $n \nearrow + \infty$, where $u \in H^1_0(\O;\R^N)$ is the solution of 
$$ \begin{cases}
- {\rm div} \big( A e(u) \big) = f \quad \text{in } \O \\
u = 0 \quad \text{on } \partial \O.
\end{cases}
$$
Given a volume fraction $\theta \in L^\infty(\O;[0,1])$ and $A,B \in L^\infty(\O;\F_{\alpha,\beta})$ with $A \leq B$ as quadratic forms on $\Ms$, the $\G$-closure set 
$$\G_\theta (A,B)$$
is defined as the set of all possible $H$-limits of $\chi_n A + (1 - \chi_n)B$ where $\chi_n \in L^\infty(\O;\{ 0,1 \})$ weakly-* converges to $\theta$ in $L^\infty(\O;[0,1])$.

\medskip

\noindent \textbf{Convex analysis.} We recall some definition and standard results from convex analysis (see \cite{Rockafellar}).
Let $f : \R^N \to [0,+ \infty]$ be a proper function ({\it i.e.} not identically $+ \infty$). The convex conjugate of $f$ is defined as 
$$f^* (x) = \sup_{y \in \R^N} \, \left\{ x \cdot y - f(y) \right\} $$
which turns out to be convex and lower semicontinuous. If $f$ is convex and finite, we define its recession function as 
$$f^\infty(x) = \lim_{t \nearrow + \infty} \, \frac{f(tx)}{t} \in [0, + \infty] $$
which is convex and positively $1$-homogeneous. If $f,g : \R^N \to [0,+ \infty]$ are proper convex functions, then their infimal convolution is defined as 
$$f \infconv g (x) = \inf_{y \in \R^N} \, \left\{ f(x-y) + g(y) \right\} $$
which is convex as well. The indicator function of a set $C \subset \R^N$ is defined as $I_C = 0$ in $ C$ and $ + \infty$ otherwise. The convex conjugate $I_C^*$ of $I_C$ is called the support function of $C$.

\section{Francfort-Garroni's model of Quasi-Static Brittle Damage} \label{section:FG}
For all $\e >0$, we consider a linearly elastic material whose reference configuration is $\O = (0,L)$, with toughness $\kappa /\e$ and stiffness tensors $\e a_0$ and $a_1$, corresponding to its damaged and sound zones respectively. Applying Theorem 2 and Remark 5 of \cite{FG} to this linearly elastic material without volume force load and with a prescribed boundary condition $w \in AC([0,T];H^1(\R))$, it ensures the following existence result for a relaxed quasi-static damage evolution.

\begin{prop}\label{prop:FG}
For all $\e >0$, there exist a time-dependent density, a displacement and a stiffness tensor 
\begin{equation}\label{eq:evolution FG eps}
 \vspace{0.1cm}
 \left\{ 
\begin{array}{l}
\vspace{0.2cm}
 \Theta_\e  : [0,T] \to L^\infty((0,L); [0,1])  \\ 
\vspace{0.2cm}
\ds  u_\e : [0,T] \to H^1((0,L))  \\ 
\ds a_\e  = \left( \frac{1-\Theta_\e}{\e a_0} + \frac{\Theta_\e}{a_1} \right)^{-1} :  [0,T] \to L^\infty((0,L))
\end{array} \right.
 \vspace{0.1cm}
\end{equation}
all weakly-* measurable, such that

\medskip
\noindent
\textbf{Dirichlet Boundary Condition:}  $  u_\e(t) \in w(t) + H^1_0((0,L)) $ for all $ t \in [0,T]$;

\medskip
\noindent
\textbf{Initial Minimality:} for all $ v \in w(0) + H^1_0((0,L))$ and $ \theta \in L^\infty((0,L); [0,1])$,
\begin{equation}\label{eq:InitialMinimality eps}
\int_0^L \left( \frac12  a_\e(0) |u_\e'(0)|^2 + \frac{\kappa}{\e} \big( 1 - \Theta_\e(0) \big) \right) \, dx 
 \leq  \int_0^L \left( \frac12 \frac{\e a_0 a_1}{\theta a_1 + (1-\theta) \e a_0} |v'|^2 + \frac{\kappa}{\e}  \theta \right) \, dx;
\end{equation}

\medskip
\noindent
\textbf{Monotonicity:} for all $T \geq t \geq s \geq 0$, $  a_\e(t)  \leq a_\e(s) $ and $ \Theta_\e(t) \leq \Theta_\e(s)$  $ \LL^1 $-a.e. in $ (0,L) $;

\medskip
\noindent
\textbf{One-sided Minimality:} for all $t \in [0,T]$, $ v \in w(t) + H^1_0((0,L))$ and $ \theta \in L^\infty((0,L);[0,1])$,
\begin{equation}\label{eq:OneSided Minimality eps}
\frac12 \int_0^L a_\e(t) \left\rvert u'_\e(t) \right\rvert ^2 \, dx \leq \frac12 \int_0^L \frac{\e a_0 a_\e(t)}{\theta a_\e(t) + (1-\theta) \e a_0} |v'|^2 \, dx + \frac{\kappa}{\e} \int_0^L \theta  \Theta_\e(t) \, dx;
\end{equation}

\medskip
\noindent
\textbf{Energy Balance: } for all $ t \in [0,T]$, the total energy
$$
\mathcal{E}_\e(t) := \frac12 \int_0^L a_\e(t) \left\rvert u'_\e(t) \right\rvert ^2 \, dx + \frac{\kappa}{\e} \int_0^L (1-\Theta_\e(t)) \, dx
$$
satisifes
\begin{equation}\label{eq:EnergyBalance eps}
\mathcal{E}_\e(t) = \mathcal{E}_\e(0) + \int_0^t \int_0^L a_\e(s)  u'_\e(s) \dot{w}'(s) \, dx ds.
\end{equation}
\end{prop}

\begin{proof}
This is the direct application of \cite[Theorem 2, Remark 5]{FG} together with \cite[Lemma 1.3.32, Formula (1.109)]{Allaire} which stipulates that for all $0 < a < b$ and all $\theta \in L^\infty((0,L);[0,1])$, 
$$
\mathcal{G}_{\theta}(a;b) = \left\{ \frac{ab}{\theta b + (1- \theta) a} \right\}.
$$
In particular, one gets that $u_\e : [0,T] \to H^1((0,L))$ is strongly measurable (in the sense of \cite[Definition 2.101]{FL}) as it is continuous outside of an at most countable subset of $[0,T]$. One also gets that $a_\e : [0,T] \to L^2((0,L))$ is strongly measurable (which is equivalent to the weak-* measurability according to Pettis' Theorem, see \cite[Theorem 2.104]{FL}), since for all $\phi \in L^2((0,L))$ 
$$t \in [0,T] \mapsto \int_0^L \phi a_\e(t)  \, dx = \int_0^L \max{(0,\phi)} a_\e(t) \, dx - \int_0^L \max{(0,-\phi)} a_\e(t)  \, dx$$
is $\LL^1$-measurable, as it is the difference between two non-increasing functions. For similar reasons, one also infers that $\Theta_\e : [0,T] \to L^2((0,L);[0,1])$ is strongly measurable. Note that $a_\e : [0,T] \to L^\infty((0,L))$ and $\Theta_\e : [0,T] \to L^\infty((0,L);[0,1])$ are (a priori) only weakly-* measurable. To see this, it suffices to take $\phi \in L^1((0,L))$ instead of $L^2((0,L))$ above.
\end{proof}

For all $\e >0$, a naive first use of the One-sided Minimality \eqref{eq:OneSided Minimality eps} entails the following properties of the stress 
$$
\sigma_\e = a_\e u'_\e : [0,T] \to L^2((0,L)).
$$
\begin{prop}
For all $\e >0$ and all $t \in [0,T]$, the stress $\sigma_\e(t) $ is homogeneous in space and $\sigma_\e  \in L^0 \big( [0,T] ; \R \big) $. Moreover,
\begin{equation}\label{eq:u'_eps}
u'_\e = \sigma_\e \left( \frac{1 - \Theta_\e}{\e a_0} + \frac{\Theta_\e}{a_1} \right) .
\end{equation}
\end{prop}
\begin{proof} 
Indeed, for all $v \in H^1_0((0,L))$ and $\delta >0$, applying \eqref{eq:OneSided Minimality eps} with $\theta = 0$ and $v^\pm := u_\e(t) \pm \delta v \in w(t) + H^1_0((0,L))$ ensures that
$$
 0 \leq \int_0^L \sigma_\e(t)u'_\e(t) \, dx \leq \int_0^L \sigma_\e(t)u'_\e(t) \, dx   \pm 2 \delta \int_0^L \sigma_\e(t) v' \, dx + \delta^2 \int_0^L a_\e(t) |v'|^2 \, dx.
$$
Dividing by $\delta>0$ then letting $\delta \searrow 0$ entails that $ \int_0^L \sigma_\e(t) v' \, dx = 0$, which implies the space homogeneity of $\sigma_\e$. Formula \eqref{eq:u'_eps} is a consequence of the expression of $a_\e$ and the definition of $\sigma_\e$.
\end{proof}

For all $\e >0$ and $t \in [0,T]$, we define the function
$$
W^t_\e : (x,\xi) \in (0,L) \times \R \mapsto \min \, \left( \frac{\kappa \Theta_\e(t)(x)}{\e} + \frac12 \e a_0 |\xi|^2 ; \frac12 a_\e(t)(x) |\xi|^2 \right).
$$
A second application of the One-Sided Minimality \eqref{eq:OneSided Minimality eps} implies that for $\LL^1$-a.e. $x \in (0,L)$
$$
\frac12 a_\e(t)(x) |u'_\e(t)(x)|^2 = \C \big(W^t_\e(x, \cdot)\big) \big(u'_\e(t)(x)\big) =: \C W^t_\e\big(x,u'_\e(t)(x)\big)
$$
where $\C W^t_\e \big( x, \cdot \big)$ is the convex envelope of $W^t_\e(x, \cdot)$.
\begin{prop}
For $ \LL^1 $-a.e. in $\{ x \in (0,L), \, a_\e(t)(x) > \e a_0 \}$,
\begin{equation}\label{eq:SQ W_eps(t)}
 \frac12 a_\e(t) |u'_\e(t)|^2 = \left\{ \begin{array}{ll}
								\vspace{0.2cm}  \ds \frac12 \sigma_\e(t) u'_\e(t) &  \ds \text{if } \frac{|\sigma_\e(t)|}{\sqrt{2 \kappa a_0}} \leq L_\e, \\ 
								\vspace{0.2cm} \ds |\sigma_\e(t)| \sqrt{\frac{2 \kappa a_0 \Theta_\e(t)}{a_\e(t) ( a_\e(t) - \e a_0 )}} - \frac{ \kappa a_0 \Theta_\e(t)}{ a_\e(t) - \e a_0}  & \ds \text{if } L_\e <  \frac{|\sigma_\e(t)|}{\sqrt{2 \kappa a_0}} \leq \frac{a_\e(t)}{\e a_0} L_\e, \\
								  \ds \frac{\kappa \Theta_\e(t)}{\e} + \frac12 \e a_0 |u'_\e(t)|^2 & \ds \text{if } \frac{a_\e(t)}{\e a_0} L_\e < \frac{|\sigma_\e(t)|}{\sqrt{2 \kappa a_0}} ,
								\end{array} \vspace{0.2cm} \right. 
\end{equation}
with 
\begin{equation}\label{eq:L_e}
L_\e := \sqrt{ \frac{a_1}{a_1 - \e a_0} } > 1.
\end{equation}
\end{prop}

\begin{proof}
Let $\e >0$ and $t \in [0,T]$. As we are working in the scalar setting, symetric quasiconvex and convex envelopes coincide. According to \cite[Lemma 3.1]{AL}, we have that for all $\xi \in \R$ and $x \in (0,L)$
\begin{multline*}
 \C  W^t_\e(x,\xi) := \inf \left\{ \int_0^1 W^t_\e(x, \xi + \phi'(y)) \, dy \, : \, \phi \in H^1_0((0,1)) \right\} \\
= \underset{\theta \in [0,1]}{\min} \left\{ \frac{\kappa \Theta_\e(t)(x)}{\e} \theta + \frac12 \left( \frac{1 - \theta}{a_\e(t)(x)} + \frac{\theta}{\e a_0} \right)^{-1} |\xi|^2 \right\} .
\end{multline*}
The above minimization being over a strictly convex function, the minimum is indeed reached at a unique minimizer in $[0,1]$. Since 
$$g : (x,\theta) \in (0,L) \times [0,1] \mapsto \frac{\kappa \Theta_\e(t)(x)}{\e} \theta + \frac12  \frac{a_\e(t)(x) \e a_0}{\theta a_\e(t)(x) + (1-\theta)\e a_0} |u'_\e(t)(x)|^2 $$
is a Carath\'eodory function, Aumann's measurable selection creterion (see \cite[Theorem 1.2]{ET}) entails that the (unique) minimizer $\theta_\e(t)(x) \in [0,1]$ of $g(x, \cdot)$ is actually $\LL^1$-measurable in $(0,L)$, {\it i.e.} $ \theta_\e(t) \in  L^\infty((0,L);[0,1])$ and for $ \LL^1$-a.e. $x \in (0,L)$
$$
 \C W^t_\e(x,u'_\e(t)(x))  = \frac{\kappa \Theta_\e(t)(x)}{\e} \theta_\e(t)(x)  + \frac12 \left( \frac{1 - \theta_\e(t)(x)}{a_\e(t)(x)} + \frac{\theta_\e(t)(x)}{\e a_0} \right)^{-1} |u'_\e(t)(x)|^2.$$
Therefore, \eqref{eq:OneSided Minimality eps} implies that
\begin{multline*}
 \frac12 \int_0^L a_\e(t) |u'_\e(t)|^2 \, dx = \underset{\theta \in L^\infty((0,L);[0,1])}{\min } \int_0^L g(x,\theta(x)) \, dx \\ \leq \int_0^L g(x,\theta_\e(t)(x)) \, dx = \int_0^L \C W^t_\e(t)(x,u'_\e(t)(x)) \, dx.
\end{multline*}
Since one simultaneously has $ \C W^t_\e(t)(x,u'_\e(t)(x)) \leq \frac12 a_\e(t)(x) |u'_\e(t)(x)|^2$ for all $x \in (0,L)$, we infer that
$$
\C W^t_\e(t)(x,u'_\e(t)(x)) = \frac12 a_\e(t)(x) |u'_\e(t)(x)|^2 \quad \text{for } \LL^1 \text{-a.e. } x \in (0,L).
$$
Therefore, \eqref{eq:SQ W_eps(t)} is the consequence of Lemma \ref{lem:Convex Envelope} (see Appendix \ref{section:appendix}) applied to $W^t_\e(x, \cdot)$ at every point $x$ in the set $\{ a_\e(t) >  \e a_0 \}$. Note that the constant $ L_\e = \sqrt{\frac{a_1}{a_1-\e a_0}} > 1$ only depends on $\e>0$ and not on $x \in (0,L)$ nor $t \in [0,T]$.
\end{proof}

\section{The limit quasi-static evolution} \label{section:effective evolution}
As previously explained, the objective of this work is to derive an effective limit model by letting $\e$ tend to $0$. Since we expect a limit model of perfect plasticity type, one has to identify which quantities will play the role of the elastic and plastic strains at the scale $\e >0$. Meanwhile, in order to pass to the limit along converging subsequences in the brittle damage evolutions described in Proposition \ref{prop:FG}, we rely on uniform bounds computed in Proposition \ref{prop:Bounded Energy} below.

\subsection{Uniform bounds and compactness}

\begin{prop}\label{prop:Bounded Energy}
There exists a constant $C(w,T) \in (0,+ \infty)$ such that
\begin{equation}\label{eq:UnifBound E esp t}
\underset{t \in [0,T]}{\sup} \, \underset{\e >0}{\sup} \, \E_\e(t)  \leq C 
\end{equation}
and
\begin{equation}\label{eq:UnifBound Theta u sigma}
 \underset{ t \in [0,T]}{\sup} \,  \underset{\e >0}{\sup} \, \left\{ \frac{1}{\e} \norme{\Theta_\e(t) - 1}_{L^1((0,L))} \, + \,  \norme{u_\e(t)}_{BV((0,L))} \, + \,   |\sigma_\e(t)|  \right\} \leq C .
\end{equation}
\end{prop}

\begin{proof} Let us first prove \eqref{eq:UnifBound E esp t}. Note that \eqref{eq:InitialMinimality eps} applied with $\theta = 0$ and $v = w(0)$ directly entails
$$
\mathcal{E}_\e(0) \leq \frac12 \int_0^L a_1 |w'(0)|^2 \, dx .
$$
As for the subsequent times $t \in [0,T]$, \eqref{eq:OneSided Minimality eps} applied with $\theta = 0$ and $v= w(t)$ entails 
$$
\frac12 \int_0^L a_\e(t) |u'_\e(t)|^2 \, dx \leq \frac12 \int_0^L a_1 |w'(t)|^2 \, dx \leq  \frac{a_1}{2} \norme{w}^2_{L^\infty([0,T];H^1(\R))} .
$$
On the other hand, since $w$ is absolutely continuous from $[0,T]$ into $H^1(\R)$, we infer that 
$$\dot{w}' \in L^1([0,T]; L^2(\R))$$
is Bochner integrable. Therefore, gathering the uniform bound on the initial time energies together with Cauchy-Schwarz inequality for the scalar-product 
$$ (\xi,\eta) \in L^2((0,L)) \times  L^2((0,L)) \mapsto  \int_0^L a_\e(s) \xi\eta \, dx  \in \R$$
(for all time $s \in [0,t]$ fixed) and the Energy Balance \eqref{eq:EnergyBalance eps}, we get that 
\begin{multline*}
\mathcal{E}_\e(t) = \mathcal{E}_\e(0) + \int_0^t \int_0^L \sigma_\e(s)\dot{w}'(s) \, dxds \\
  \leq \E_\e(0) +  \int_0^t \left( \int_0^L a_\e(s) |u'_\e(s)|^2 \, dx \right)^\frac12 \left( \int_0^L a_\e(s) |\dot{w}'(s)|^2 \, dx \right)^{\frac12} \, ds   \leq  C(w,T) .
\end{multline*}

\medskip
We next show \eqref{eq:UnifBound Theta u sigma}. Let $\e >0$ and $t \in [0,T]$. First, as shown in \cite[Lemma 2.3, Formula (2.6)]{BIR}, there exists a constant $c >0 $ (only depending on $a_0$, $ a_1$ and $\kappa$) such that the function
$$
W_\e : \xi \in \R \mapsto \min \, \left\{ \frac{\kappa}{\e} + \frac12 \e a_0 |\xi|^2 ; \frac12 a_1 |\xi|^2 \right\}
$$
satisfies $ \C W_\e(\xi) \geq c | \xi | - \frac1c$ for all $\xi \in \R$. Remembering the definition \eqref{eq:evolution FG eps} of $a_\e(t) $ and the fact that
$$ \C W_\e(\xi) = \underset{\theta \in [0,1]}{\min} \left\{ \frac{\kappa \theta}{\e} + \frac12 \left( \frac{\theta}{\e a_0} + \frac{1-\theta}{a_1} \right)^{-1} |\xi|^2 \right\}$$
for all $\xi \in \R$, we in particular get that
$$\E_\e(t)  \geq \int_0^L \C W_\e \big(u'_\e(t) \big) \, dx \geq  c \norme{u'_\e(t)}_{L^1((0,L))} - \frac{L}{c} .
$$
Thus, using the equivalent norm in $BV((0,L))$ recalled in \eqref{eq:norme BV} leads to
$$\norme{u_\e(t)}_{BV((0,L))} \leq \norme{u_\e(t)'}_{L^1((0,L))} +  | w(t)(0)| + |w(t)(L)| \leq C(w,T).$$
Finally, the homogeneity in space of the stress $\sigma_\e(t) \in \R$ implies that
$$
\E_\e(t) \geq \frac12 \int_0^L a_\e(t) |u'_\e(t)|^2 \, dx \geq \frac{T}{2 a_1} |\sigma_\e(t)|^2
$$
as well, thus concluding \eqref{eq:UnifBound Theta u sigma}.
\end{proof}

From the uniform bounds \eqref{eq:UnifBound Theta u sigma}, we obtain compactness properties. 
\begin{prop}\label{cor:subsequence}
There exists a subsequence (not relabelled and independent of $t$) and a non-negative Radon measure $ \mu : [0,T] \to \M \big( [0,L];\R^+ \big)$, which is non-decreasing in time, such that
\begin{subequations}
	\begin{empheq}[left=\empheqlbrace]{align}
	\, & \Theta_\e(t) \to 1 \quad \text{strongly in } L^1((0,L)) \text{ for all } t \in [0,T] \label{eq:CV Theta eps to 1} \\
	\, & \mu_\e(t) := \frac{1 - \Theta_\e(t) }{\e} \mathds{1}_{(0,L)} \rightharpoonup  \mu(t) \quad \text{weakly-* in } \M \big( [0,L]) \text{ for all } t \in [0,T] \label{eq:mu eps to mu}
	\end{empheq}
\end{subequations}
as $\e \searrow 0$.
\end{prop}

\begin{rem}
Henceforth, we will work along the subsequence (not relabeled) mentionned in Proposition \ref{cor:subsequence}.
\end{rem}

\begin{proof}
One directly deduces from \eqref{eq:UnifBound Theta u sigma} that $\Theta_\e (t) $ strongly converges to $1$ in $L^1((0,L))$ as $\e \searrow 0$ for all time $t \in [0,T]$. Next, using the non-decreasing character (in time) of the non-negative Radon measures 
$$
\mu_\e : t \in [0,T] \mapsto \frac{1 - \Theta_\e(t) }{\e} \mathds{1}_{(0,L)} \in \M \big( [0,L];\R^+ \big)
$$
together with \eqref{eq:UnifBound Theta u sigma}, one can apply the generalized version of Helly's Theorem recalled in \cite[Lemma 7.2]{DMDSM} for the topological duals of separable Banach spaces and find a subsequence and a limit time-dependent non-negative Radon measure $\mu : [0,T] \to \M \big( [0,L];\R^+ \big)$ such that $\mu_\e(t) \rightharpoonup \mu(t)$ weakly-* in $\M \big( [0,L] \big)$ as $\e$ tends to $0$, for all $t \in [0,T]$. Note that the monotonicity in time of $\mu_\e$ is preserved by the pointwise weak-* convergence in $\M \big( [0,L] \big)$.
\end{proof}

\subsection{Initial time of the evolution}
We begin with a corollary of the analysis led in \cite[Theorem 3.1]{BIR} in the static setting, taking into account a prescribed Dirichlet boundary datum. We refer to Appendix \ref{section:appendix} for a more general statement of this proposition and its proof.
\begin{prop}
The functional $ \F_\e : L^1((0,L)) \to \bar{ \R^+}$ defined for all $u \in L^1((0,L))$ by
$$
 \F_\e(u) :=  \begin{cases}
{\displaystyle \int_0^L  \C W_\e(u') \, dx } & \quad \text{ if } u \in w(0) + H^1_0((0,L)) \\
+ \infty & \quad \text{ otherwise}
\end{cases}
$$
$\Gamma$-converges in $L^1$ as $\e \searrow 0$ to the functional
$$
 \F(u) := \begin{cases}
 \begin{array}{r}
\ds \int_0^L \bar W(u') \, dx + \sqrt{2\kappa a_0} \left\lvert D^s u \right\rvert \big( (0,L) \big) + \sqrt{2\kappa a_0} \big( \left\lvert w(0)(L) -u(L) \right\rvert + \left\lvert w(0)(0) -u(0) \right\rvert \big)  \\
 \quad \text{ if } u \in BV((0,L))
\end{array} \\
+ \infty \quad \text{ otherwise}
\end{cases}
$$
where
$$ \bar W : \xi \in \R \mapsto \underset{\eta \in \R}{\inf} \left\{ \frac{a_1}{2} \left\lvert \xi - \eta \right\rvert^2 + \sqrt{2 \kappa a_0} \left\lvert \eta \right\rvert \right\} .$$
\end{prop}

\medskip

In particular, we infer that 
$$
\E_\e(0) = \F_\e(u_\e(0)) = \min \F_\e.
$$
Indeed, according to \cite[Lemma 3.1]{AL} we first remark that 
$$\C W_\e(u'_\e(0)) \leq \frac{\kappa}{\e} (1 - \Theta_\e(0)) + \frac{a_\e(0)}{2} \left\lvert u'_\e(0) \right\rvert^2$$
hence $\F_\e(u_\e(0)) \leq \E_\e(0)$. Besides, \eqref{eq:InitialMinimality eps} implies that $\E_\e(0) \leq \F_\e(v)$ for all $v \in w(0) + H^1_0((0,L))$, where we exchanged the infimum and the integral thanks to Aumann's criterion.

\medskip

Therefore, we deduce from the Fundamental Theorem of $\Gamma$-convergence together with the bounds \eqref{eq:UnifBound E esp t} and \eqref{eq:UnifBound Theta u sigma} that there exist a further subsequence (still not relabeled) and a displacement $ u(0) \in BV((0,L))$ such that
\begin{equation}\label{eq:Initial Time}
u_\e(0) \rightharpoonup u(0) \text{ weakly-* in } BV((0,L)) \quad \text{and} \quad \E_\e(0) \to \E(0) := \F(u(0)) = \min \F
\end{equation}
when $\e \searrow 0$.

\subsection{Time independence of the subsequences}
We first show that, along the whole subsequence introduced in Proposition \ref{cor:subsequence} (not relabeled and independent of $t$), $\{ \sigma_\e(t) \}_\e$ pointwise converges to some limit stress $ \sigma(t)$ for all time in $[0,T]$. 

\begin{prop}\label{cor:Sigma=Sigma_t}
Along the subsequence introduced in Proposition \ref{cor:subsequence} (independent of $t$, not relabeled), we have that for all time $t \in [0,T]$,
$$
\mu_\e(t) \big( [0,L] \big) =  \int_0^L \frac{1 - \Theta_\e(t)}{\e} \, dx \to l(t):= \mu(t) \big( [0,L ] \big) \geq 0
$$
and
\begin{equation}\label{eq:2}
\sigma_\e(t) \to \sigma (t) := \frac{ \big[ w(t) \big]^L_0 }{ \frac{l(t)}{a_0} + \frac{L}{a_1}} 
\end{equation}
when $ \e \searrow 0$, where $\big[ w(t) \big]^L_0 := w(t)(L) - w(t)(0)$. Moreover, 
\begin{equation}\label{eq:sigma in KK}
\sigma (t) \in \KK = [- \sqrt{2 \kappa a_0}, \sqrt{2 \kappa a_0} ] \quad \text{for all $t \in [0,T]$}. 
\end{equation}
In particular, we infer that $ \sigma \in L^\infty([0,T];\KK)$,
\begin{equation}\label{eq:CV faible * sigma}
\sigma_\e \overset{*}{\rightharpoonup} \sigma \quad \text{weakly-* in } L^\infty([0,T];\R) 
\end{equation} 
and for all $t \in [0,T]$,
\begin{equation}\label{eq:E(t)}
\E_\e(t) \underset{\e \searrow 0}{\longrightarrow} \E(0) + \int_0^t \int_0^L \sigma(s) \dot{w}'(s)(x) \, dx ds = \frac{\big[ w(t)\big]^L_0}{2}  \sigma (t) + \kappa l(t) =: \E(t) .
\end{equation}
\end{prop}

\begin{proof}
We work along the subsequence introduced in Proposition \ref{cor:subsequence}. Let $t \in [0,T]$. Note that by the weak-* convergence of $\mu_\e(t)$ to $\mu(t)$ in $\M([0,L])$, one gets that
$$
\mu_\e(t) \big( [0,L] \big) = \int_0^L \frac{1 - \Theta_\e(t)(x)}{\e} \, dx \underset{\e \searrow 0}{\longrightarrow} l(t):= \mu(t) \big( [0,L] \big) \geq 0.
$$
Since $u_\e(t) \in w(t) + H^1_0((0,L))$ satisfies $u'_\e(t) = \sigma_\e(t) a_\e(t)^{-1}$ and $\sigma_\e(t)$ is homogeneous in space, we infer by the Integration by Parts Formula in $H^1((0,L))$ that 
$$
\big[w(t)\big]^L_0 = \int_0^L u'_\e(t)(x) \, dx = \sigma_\e(t) \int_0^L \left( \frac{1- \Theta_\e(t)(x)}{\e a_0} + \frac{\Theta_\e(t)(x)}{a_1} \right) \, dx .
$$
Since 
$$
\int_0^L \left( \frac{1- \Theta_\e(t)(x)}{\e a_0} + \frac{\Theta_\e(t)(x)}{a_1} \right) \, dx  \underset{\e \searrow 0}{\longrightarrow}   \frac{l(t)}{a_0} + \frac{L}{a_1} > 0,
$$
we obtain
$$
\sigma_\e(t) \underset{\e \searrow 0}{\longrightarrow} \sigma(t) := \frac{\big[ w(t) \big]^L_0 }{ \frac{l(t)}{a_0} + \frac{L}{a_1} },
$$
which proves \eqref{eq:2} and \eqref{eq:CV faible * sigma}. Next, let us remark that for all $\e >0$ and $\LL^1$-a.e. in $\{ a_\e(t) > \e a_0 \}$,
$$
\frac12 \sigma_\e(t) u_\e'(t) > \left\lvert \sigma_\e(t) \right\rvert \sqrt{ \frac{2 \kappa a_0 \Theta_\e(t)}{a_\e(t) (a_\e(t) - \e a_0)}} - \frac{\kappa a_0 \Theta_\e(t)}{a_\e(t) - \e a_0} \quad \text{if } \frac{\left\lvert \sigma_\e(t) \right\rvert}{\sqrt{2 \kappa a_0}} > L_\e.
$$
Indeed, 
$$\frac12 \sigma_\e(t) u_\e'(t) - \left\lvert \sigma_\e(t) \right\rvert \sqrt{ \frac{2 \kappa a_0 \Theta_\e(t)}{a_\e(t) (a_\e(t) - \e a_0)}} + \frac{\kappa a_0 \Theta_\e(t)}{a_\e(t) - \e a_0} = \left( \sqrt{\frac{a_\e(t)}{2}} \left\lvert u_\e'(t) \right\rvert - \sqrt{ \frac{\kappa a_0  \Theta_\e(t)}{a_\e(t) - \e a_0} } \right)^2 \geq 0$$
and the equality holds if and only if $\left\lvert \sigma_\e(t) \right\rvert = \sqrt{2 \kappa a_0} L_\e$. Thus, by homogeneity in space of $\sigma_\e$ and \eqref{eq:SQ W_eps(t)} we infer that $ \left\lvert \sigma_\e(t) \right\rvert \leq \sqrt{2 \kappa a_0} L_\e $ for all $ \e > 0 $ such that $ \{ a_\e(t) > \e a_0 \} \neq \emptyset$.
\begin{itemize}
\item Either $\limsup_{\e \searrow 0} \LL^1 \left( \{ a_\e(t) > \e a_0 \} \right) = 0$, hence 
$$
\quad \quad \quad L \left\lvert  \sigma (t) \right\rvert = \lim_\e L \left\lvert \sigma_\e(t) \right\rvert = \lim_\e \left( \int_{\{ a_\e(t) = \e a_0 \}} \e a_0 \left\lvert u_\e'(t) \right\rvert \, dx + \left\lvert \sigma_\e(t) \right\rvert \LL^1 \left( \{ a_\e(t) > \e a_0 \} \right)  \right) = 0
$$
and $ \sigma (t) = 0 \in \KK$.
\item Or $\limsup_{\e \searrow 0 }  \LL^1 \left( \{ a_\e(t) > \e a_0 \} \right) > 0$. Up to another subsequence (still not relabeled), we can assume that $\LL^1 \left( \{ a_\e(t) > \e a_0 \} \right) > 0$ for all $\e > 0$. In particular, it implies that $\{ a_\e(t) > \e a_0 \} \neq \emptyset$ hence $\left\lvert \sigma_\e(t) \right\rvert \leq \sqrt{2 \kappa a_0} L_\e$ for all $\e > 0$. Taking the limit when $\e \searrow 0$, we obtain that $ \left\lvert  \sigma (t) \right\rvert \leq \sqrt{2 \kappa a_0}$ hence $ \sigma (t) \in \KK$.
\end{itemize}
In any cases, we get that the stress constraint \eqref{eq:sigma in KK} is satisfied. Thus, by homogeneity in space of $\sigma_\e(t)$, the Dominated Convergence Theorem entails that
$$
\int_0^t \int_0^L \sigma_\e(s) \dot{w}'(s)(x) \, dx ds = \int_0^t \sigma_\e(s) \left( \int_0^L \dot{w}'(s)(x) \, dx \right) \, ds \underset{\e \searrow 0}{ \longrightarrow} \int_0^t \int_0^L  \sigma(s) \dot{w}'(s)(x) \, dx ds.
$$
Therefore, using the convergence of the energies at the initial time \eqref{eq:Initial Time} together with the Energy Balance \eqref{eq:EnergyBalance eps}, one gets that
$$
\E_\e(t)  \underset{\e \searrow 0}{\longrightarrow} \E(0) + \int_0^t \int_0^L \sigma(s) \dot{w}'(s)(x) \, dx ds =: \E(t) \in \R^+.
$$
By homogeneity in space of $\sigma_\e(t)$ again, since
$$
{\displaystyle
\E_\e(t) = \frac12 \sigma_\e(t) \int_0^L w'(t)(x) \, dx + \kappa \int_0^L \frac{1 - \Theta_\e(t)(x)}{\e} \, dx \underset{\e \searrow 0}{\longrightarrow}   \frac{\big[ w(t)\big]^L_0}{2}  \sigma (t) + \kappa l(t) }
$$
we deduce that
$$
\E(t) = \frac{\big[ w(t)\big]^L_0}{2}  \sigma (t) + \kappa l(t)
$$
which completes the proof of \eqref{eq:E(t)} and Proposition \ref{cor:Sigma=Sigma_t}.
\end{proof}

We now define what will play the role of the elastic and plastic strains at the scale $\e >0$, by setting
\begin{subequations}
\begin{empheq}[left=\empheqlbrace]{align}
	& e_\e = \sigma_\e \frac{\Theta_\e}{a_1} \in L^2((0,L)) \label{seq:e_eps}\\
	& p_\e = \frac{\sigma_\e}{a_0} \mu_\e = \sigma_\e \frac{1 - \Theta_\e}{\e a_0} \mathds{1}_{(0,L)} \in L^2((0,L)) \label{seq:p_eps}
	\end{empheq}
\end{subequations}
which, by \eqref{eq:u'_eps}, satisfy the additive decomposition $u'_\e = e_\e + p_\e $ at all time. Using Proposition \ref{cor:subsequence} together with the homogeneity in space of $\sigma_\e$ and $\sigma$, we infer that for all $t \in [0,T]$ 
\begin{subequations}
\begin{empheq}[left=\empheqlbrace]{align}
	& e_\e(t) \to e(t) := \frac{ \sigma(t)}{a_1} \quad \text{strongly in } L^2((0,L)) \label{seq:def e}\\
	& p_\e(t)  \rightharpoonup p(t) := \frac{\sigma(t)}{a_0} \mu(t) \quad \text{weakly-* in } \M \big( [0,L] \big)  \label{seq:def p} \\
	& u'_\e(t) \rightharpoonup \sigma(t) \left( \frac{\mu(t) \res (0,L)}{a_0} + \frac{ \LL^1 }{a_1} \right) = p(t) \res (0,L) + e(t) \LL^1  \quad \text{weakly-* in } \M ((0,L)) \label{seq:u' eps to Du}
	\end{empheq}
\end{subequations}
when $\e \searrow 0$. Let us recall that 
\begin{equation}\label{eq:link mu l}
\mu(t)([0,L]) = l(t) \quad \text{and} \quad p(t)([0,L]) = \frac{\sigma(t)}{a_0} l(t).
\end{equation}

\medskip
The uniform bound \eqref{eq:UnifBound Theta u sigma} ensures that for all $t \in [0,T]$, there exist a further subsequence (depending on $t$, not relabeled) and a displacement $u(t) \in BV((0,L))$ such that
$$
u_\e(t) \rightharpoonup u(t) \quad \text{weakly-* in } BV((0,L))
$$
when $\e \searrow 0$. In particular, we deduce from \eqref{seq:u' eps to Du} that 
$$
Du(t) = p(t) \res (0,L) + e(t) \LL^1  \quad \text{in } \M((0,L);\R)
$$ 
is independent of the subsequence defining $u(t)$. Moreover, one can check that
\begin{equation}\label{eq:relaxed boundary condition}
p(t) \res \{ 0,L \} = \big(w(t) - u(t) \big) \big( \delta_L - \delta_0 \big) .
\end{equation}
Indeed, extending the problem on a larger open interval $[0,L] \subset \O'$ and setting 
\begin{empheq}[left=\empheqlbrace]{align}
	& \bar u_\e(t) = u_\e(t) \mathds{1}_{(0,L)} + w(t) \mathds{1}_{\O' \setminus (0,L)} \nonumber \\
	& \bar p_\e(t) = p_\e(t) \mathds{1}_{[0,L]} \nonumber \\
	& \bar e_\e(t) = e_\e(t) \mathds{1}_{(0,L)} + w'(t) \mathds{1}_{\O' \setminus (0,L)} \nonumber
	\end{empheq}
one can check that 
\begin{empheq}[left=\empheqlbrace]{align}
	& \bar u_\e(t) \rightharpoonup u(t) \mathds{1}_{ (0,L)} +	w(t) \mathds{1}_{ \O' \setminus (0,L)} \quad \text{weakly-* in } BV(\O') \nonumber \\
	& \bar p_\e(t) \rightharpoonup \bar p(t) : = p(t) \mathds{1}_{[0,L]}  \quad \text{weakly-* in } \M(\O') \nonumber \\
	& \bar e_\e(t) \rightharpoonup \bar e(t) := e(t) \mathds{1}_{ (0,L)}  + w'(t)  \mathds{1}_{ \O' \setminus (0,L)} \quad \text{weakly in } L^2(\O') \nonumber
	\end{empheq}
when $\e \searrow 0$. Therefore, using \cite[Remark 2.3 (i)]{Temam}, we get that
$$
D\bar u_\e(t) = \bar e_\e(t) + \bar p_\e(t) 
 \rightharpoonup \bar e(t) + \bar p (t) = Du(t) \mathds{1}_{(0,L)} + w'(t)  \mathds{1}_{ \O' \setminus (0,L)} \LL^1 + \big( w(t) - u(t) \big) \big( \delta_L - \delta_0 \big)
$$
which implies \eqref{eq:relaxed boundary condition}. In particular, we infer that the limit displacement $u(t)$ is actually independent of the subsequence. Indeed, let $u_1(t)$ and $u_2(t) \in BV((0,L))$ be two weak limits of $\{ u_\e(t) \}_{\e > 0}$ in $BV((0,L))$. On the one hand, \eqref{seq:u' eps to Du} entails that $
D(u_1(t) - u_2(t)) = 0 \quad \text{in } \M \big( (0,L) \big) 
$ hence $u_1(t) - u_2(t) = C(t) \in \R $ is homogeneous in space. On the other hand, the internal traces of $u_1(t)$ and $u_2(t)$ being prescribed on $\{ 0,L \}$ by \eqref{eq:relaxed boundary condition}, we infer that $C(t) = \big(u_1(t) - u_2(t) \big) (L) = 0$ hence $u_1(t) = u_2(t)$ in $BV((0,L))$. Therefore, we infer that the whole sequence converges and there exists 
$$u : [0,T] \to BV((0,L))$$
such that
\begin{equation}\label{eq:u}
u_\e(t) \rightharpoonup u(t) \quad \text{weakly-* in } BV((0,L)) \text{ when } \e \searrow 0, \text{ for all } t \in [0,T].
\end{equation}

\medskip   

Note that $u(0) \in BV((0,L))$ was already given by \eqref{eq:Initial Time} and the static analysis led in \cite{BIR}, entailing the following Constitutive Equation at the initial time.
\begin{prop}
\begin{equation}\label{eq:E 0}
\E(0) = \frac{La_1}{2} e(0)^2 + \sqrt{2\kappa a_0} \left\lvert p(0) \right\rvert \big( [0,L ] \big)
\end{equation}
and
\begin{equation}\label{eq:Constitutive Equation Initial Time}
\left( 2 \kappa a_0  -  \sigma(0)^2 \right) l(0) = 0.
\end{equation}
\end{prop}

\begin{proof}
Gathering \eqref{seq:u' eps to Du} and \eqref{eq:relaxed boundary condition}, we can identify the absolutely continuous and singular parts (with respect to $\LL^1$) in the Radon-Nikod\'ym decompositions of $Du(0)$ and $p(0)$:
$$
\begin{cases}
\, D^a u(0) = e(0) \LL^1 + p(0)^a \\
\,  p(0)^s = D^s u(0) \mathds{1}_{(0,L)} + (w(0) - u(0)) \left( \delta_L - \delta_0 \right).
 \end{cases}
$$
By definition of the inf-convolution and \eqref{eq:Initial Time}, we get that $\bar W(D^a u(0)) \leq \frac{a_1}{2} e(0)^2 + \sqrt{2 \kappa a_0} \left\lvert p(0)^a \right\rvert $ $\LL^1$-a.e. in $(0,L)$ and $ \E(0) = \F(u(0)) \leq \frac{La_1}{2} e(0)^2 + \sqrt{2\kappa a_0} \left\lvert p(0) \right\rvert \big( [0,L ] \big)$, where we identified absolutely continuous measures with their densities. Besides, combining \eqref{eq:E(t)}, \eqref{eq:2} and \eqref{eq:link mu l}, we also have that 
\begin{multline*}
 \E(0) = \frac{\sigma(0)^2}{2} \left( \frac{l(0)}{a_0} + \frac{L}{a_1} \right) + \kappa l(0) \\
 = \frac{La_1}{2} e(0)^2 + \sqrt{2\kappa a_0} \left\lvert p(0) \right\rvert \big( [0,L ] \big) + \left( \sqrt{2 \kappa a_0} - \left\lvert \sigma(0) \right\rvert \right)^2 \frac{l(0)}{2 a_0} \\
 \geq \frac{La_1}{2} e(0)^2 + \sqrt{2\kappa a_0} \left\lvert p(0) \right\rvert \big( [0,L ] \big).
\end{multline*}
Therefore, we obtain that $ \E(0) = \frac{La_1}{2} e(0)^2 + \sqrt{2\kappa a_0} \left\lvert p(0) \right\rvert \big( [0,L ] \big)$ and $ \left( 2 \kappa a_0  -  \sigma(0)^2 \right) l(0) = 0$.
\end{proof}

\medskip

\subsection{Regularity of the evolution}
Looking at the proof of Proposition \ref{cor:Sigma=Sigma_t} and defining, for all time $t \in [0,T]$, the function
$$
\Delta(t) = \left( \frac{\E(t)}{a_0} + \frac{\kappa L}{a_1} \right)^2 - \frac{2\kappa}{a_0} \left\lvert \big[ w(t) \big]^L_0 \right\rvert^2,
$$
one can check that 
$$\Delta (t) = \left( \frac{\sigma(t)^2}{2a_0} - \kappa \right)^2 \left( \frac{l(t)}{a_0} + \frac{L}{a_1} \right)^2 \geq 0 \quad \text{and} \quad l(t) = \frac{a_0}{2\kappa} \left( \frac{\E(t)}{a_0} - \frac{\kappa L}{a_1} + \sqrt{ \Delta(t) } \right) .$$
On the one hand, since 
$$\E(t) = \E(0) + \int_0^t \int_0^L \sigma(s) \dot w'(s,x) \, dx ds$$
and 
$$s \mapsto  \int_0^L \sigma(s) \dot w'(s,x) \, dx \quad \text{is integrable on } [0,T],$$
we infer that $\E \in AC([0,T];\R)$. On the other hand, since $w \in AC \big( [0,T]; C^0([0,L]) \big)$ and the product of absolutely continuous functions remains absolutely continuous, we infer that $\Delta \in AC([0,T];\R)$. In particular, we deduce that $l$, $\sigma$, $e$ and $p$ are continuous on the whole interval $[0,T]$. By non-negativity and monotonicity in time of $\mu$, we also infer that $\mu$ is continuous from $[0,T]$ to $\M([0,L])$. 

\medskip
Using the Energy Balance in Proposition \ref{prop:FG} together with the non-decreasing character of $l$, \eqref{eq:E(t)} and \eqref{eq:2}, we can actually show that $\sigma \in AC([0,T];\R)$. From this, we will deduce that $l$, $u$, $p$ and $\mu$ inherit the same regularity. This is a strong result because it is usually obtained a posteriori, once an Energy Balance of the type \eqref{eq:Energy Balance} is proved to be satisfied (see \cite{DMDSM}). Remarkably, here we do not rest on such an Energy Balance in order to prove the regularity of the quasi-static evolution. This is the content of the following proposition.
\begin{prop}
The mapping $(\sigma,e,l,\mu,p,u) : [0,T] \to \KK \times \R  \times  \R^+ \times  \M([0,L]) \times  \M([0,L]) \times  BV((0,L)) $ is absolutely continuous.
\end{prop}
\begin{proof}
One can check that for all $0 \leq s \leq t \leq T$,
\begin{multline}\label{eq:Imitation DSDMMORA}
\frac{L}{2 a_1} \left( \sigma(t) - \sigma(s) \right)^2 = \int_s^t (\sigma(r)  - \sigma(s) ) \big[ \dot w(r)  \big]^L_0 \, dr \\
- \frac{ l(s) }{2 a_0} (\sigma(t) - \sigma(s))^2   - \frac{ (l(t) - l(s))}{2 a_0} \left( \sigma(t) ^2 - 2 \sigma(t) \sigma(s) + 2 \kappa a_0 \right). 
\end{multline}
Indeed, the convergences in Proposition \ref{cor:Sigma=Sigma_t} together with the homogeneity in space of $\sigma$ and $\sigma_\e$ imply that
\begin{multline*}
\frac{L}{2a_1} (\sigma(t) - \sigma(s) )^2 = L \left( \frac{\sigma(t) e(t)}{2} - \frac{ \sigma(s) e(s)}{2} - \sigma(s) (e(t) - e(s)) \right) \\
= \lim_{\e \searrow 0} \int_0^L \left( \frac{\sigma_\e(t) e_\e(t)}{2} - \frac{ \sigma_\e(s) e_\e(s)}{2} - \sigma_\e(s) (e_\e(t) - e_\e(s)) \right) \, dx .
\end{multline*}
Using \eqref{seq:def e}, \eqref{seq:e_eps}, \eqref{seq:p_eps} and writing $l_\e := \int_0^L \frac{1 - \Theta_\e}{\e} \, dx$, we get that
\begin{align*}
\frac{L}{2a_1} (\sigma(t) - \sigma(s) )^2 
& = \lim_{\e \searrow 0}  \biggl( \E_\e(t) -\int_0^L \frac{\sigma_\e(t) p_\e(t)}{2} \, dx - \kappa l_\e(t)  - \E_\e(s) + \int_0^L  \frac{ \sigma_\e(s) p_\e(s)}{2} \, dx  + \kappa l_\e(s) \\ 
& \hspace{4cm} - \sigma_\e(s) \int_s^t \big[ \dot w(r) \big]^L_0 \, dr + \sigma_\e(s) \int_0^L (p_\e(t) - p_\e(s)) \, dx \biggr) \\
& = \E(t) - \E(s) - \kappa (l(t) - l(s))  - \sigma(s) \int_s^t \big[ \dot w(r) \big]^L_0 \, dr \\
& \hspace{4cm} - \frac{\sigma(t)^2}{2 a_0} l(t) + \frac{ \sigma(s)^2}{2a_0} l(s) + \sigma(s) \left( \frac{ \sigma(t) l(t)}{a_0} - \frac{\sigma(s) l(s)}{a_0} \right)
\end{align*}
so that \eqref{eq:E(t)} entails \eqref{eq:Imitation DSDMMORA}. In particular, since $2 \kappa a_0 \geq \sigma(s)^2$, we obtain
\begin{multline*}
\frac{L}{2a_1} \left\lvert \sigma(t) - \sigma(s) \right\rvert^2 \leq \int_s^t (\sigma(r)  - \sigma(s) ) \big[ \dot w(r) \big]^L_0 \, dr \\
\leq   \frac{L}{2a_1} \int_s^t \left\lvert \sigma(r) - \sigma(s) \right\rvert^2 \, dr + \frac{a_1}{2L} \int_0^T  \left\lvert \big[ \dot w (r) \big]^L_0 \right\rvert^2 \, dr 
\end{multline*}
where we also used Young's inequality. Applying Gronwall's Lemma leads to
$$
\left\lvert \sigma(t) - \sigma(s) \right\rvert \leq  \frac{a_1}{L} \norme{ \big[ \dot w \big]^L_0 }_{L^2([0,T])} \exp \left( \frac{t-s}{2} \right)
$$
for all $0 \leq s \leq t \leq T$, thus proving the absolute continuity of $\sigma$ and $e= \sigma a^{-1}_1 $ from $ [0,T]$ to $ \R$. In particular, \eqref{eq:E(t)} entails that 
$$ l = \frac{1}{\kappa} \left( \E - \frac{\sigma}{2} \big[ w \big]^L_0 \right) \in AC([0,T];\R).$$
Since $\mu : [0,T] \to \M([0,L];\R^+)$ is non-decreasing in time, \eqref{eq:link mu l} implies  
$$
0 \leq \left( \mu(t) - \mu(s) \right) \big( [0,L] \big) = l(t) - l(s)
$$
for all $0 \leq s \leq t \leq T$, so that 
$$\mu \in AC \big([0,T];\M([0,L]) \big) \quad \text{and} \quad p = \frac{\sigma}{a_0} \mu \in AC \big([0,T];\M([0,L]) \big).$$
Finally, the Trace Theorem in $BV((0,L))$, \eqref{eq:u} and \eqref{eq:relaxed boundary condition} imply that for all $0 \leq s \leq t \leq T$,
$$
\norme{u(t) - u(s)}_{BV((0,L))} \leq  L \left\lvert e(t) - e(s) \right\rvert + \left\lvert p(t) - p(s) \right\rvert ([0,L])  
 + \left\lvert w(t)-w(s) \right\rvert (L) + \left\lvert w(t) - w(s) \right\rvert (0), 
$$
thus proving the absolute continuity of $u$ from $[0,T]$ to $BV((0,L))$ as well.
\end{proof}

At this point, we have identified good candidates for the limit evolution:
$$(u,e,p,\sigma, \mu) : [0,T] \to BV((0,L)) \times \R \times \M([0,L]) \times \KK \times \M([0,L]),$$
which are all absolutely continuous on $[0,T]$ and satisfy the following assertions for all $t \in [0,T]$:
\begin{enumerate}[label=\roman*., leftmargin=* ,parsep=0.1cm,topsep=0.2cm]
\item Additive Decomposition: \quad $Du(t) = e(t)\LL^1 \res (0,L) + p(t) \res (0,L)$ in $\M((0,L))$
\item Relaxed Dirichlet Condition: \quad $p(t) \res \{ 0,L \} = \big(w(t) - u(t) \big) \big( \delta_L - \delta_0 \big)$ in $\M(\{0,L \})$ 
\item Constitutive Equation: \quad $\sigma(t) = a_1 e(t) $								
\item Equilibrium Equation: \quad $\sigma'(t) = 0$ in $H^{-1}((0,L))$
\item Stress Constraint: \quad $\sigma(t) \in \KK$.
\end{enumerate} 
The absolute continuity of $(u,\sigma,p,\mu,l)$ guarantees that $(u,\sigma)$ describes a quasi-static damage evolution, whose internal variable is the effective compliance (inverse effective rigidity) 
$$
c : t \in [0,T] \mapsto \frac{\mu(t)}{a_0} + \frac{1}{a_1} \LL^1 \res (0,L) \in \M([0,L];\R^+)
$$
satisfying the Constitutive Equation and Griffith Evolution Law stated in Theorem \ref{thm:DAMAGE}.

\begin{proof}[Proof of Theorem \ref{thm:DAMAGE}:]
Using \eqref{eq:E(t)} and \eqref{eq:2}, one can check that for $\LL^1$-a.e. $t \in [0,T]$ the following quantities are well defined and satisfy 
$$
 \big[ \dot w(t) \big]^L_0 = \dot \sigma (t) \left( \frac{l(t)}{a_0} + \frac{L}{a_1} \right) + \sigma(t) \frac{ \dot l(t)}{a_0}
$$
and 
$$
\dot \E(t) = \sigma(t) \big[ \dot w(t) \big]^L_0 = \frac{ \dot \sigma(t) \big[w(t)\big]^L_0 }{2} + \frac{\sigma(t) \big[ \dot w(t) \big]^L_0}{2} + \kappa \dot l(t).
$$
Hence 
$${\displaystyle \kappa \dot l(t) + \frac{\dot \sigma(t) \big[w(t)\big]^L_0}{2} = \frac{ \sigma(t) \big[ \dot w(t) \big]^L_0}{2} =  \frac{\dot \sigma(t) \big[w(t)\big]^L_0}{2} + \frac{\sigma(t)^2 \dot l(t)}{2 a_0}},$$
entailing that
$$
\dot l(t) \left( \kappa - \frac{\sigma(t)^2}{2 a_0} \right) = 0.
$$
Besides, using that $\mu$ is non-decreasing in time together with \cite[Theorem 7.1, Formula (7.4)]{DMDSM} and \eqref{eq:link mu l} ensures that 
\begin{equation}\label{eq:dot mu dot l}
 \dot \mu(t)  \big([0,L]) =   \lim_{h \searrow 0}  \frac{\mu (t+h)\big( [0,L] \big) - \mu(t)\big( [0,L] \big)}{h}    = \lim_{h \searrow 0}  \frac{l (t+h)- l (t)}{h}  = \dot l(t)
\end{equation}
for $\LL^1$-a.e. $t \in [0,T]$. Thus, by non-negativity of the Radon measure $\dot \mu(t)$, we infer that 
$$\dot \mu(t) \left( 2 \kappa a_0 - \sigma(t)^2 \right) = 0 \quad \text{in } \M([0,L];\R^+) \quad \text{for } \LL^1 \text{-a.e. } t \in [0,T]. $$
\end{proof}

\begin{rem} The effective limit model obtained here is a different type of damage model where the dissipative phenomena is described by means of an internal variable, the effective compliance $c : [0,T] \to \M([0,L];\R^+)$, whose non-decreasing character in time accounts for the irreversibility of damage. This is a threshold stress model, based on the conjecture that damage propagates if and only if the stress saturates the constraint.
\noindent
\begin{itemize}
\item Formally inverting the compliance $c(t)$, the Constitutive Equation allows us to interpret $\frac12 \sigma(t) Du(t) = \frac12 c(t)^{-1}  Du(t)^2$ as the stored (elastic) energy density of the effective damaged medium at time $t \in [0,T]$. In other words, one can interpret 
$$\mathcal Q(t) := \frac12 \sigma(t) Du(t) ((0,L)) = \frac{\sigma(t)^2}{2} \left( \frac{\mu(t)((0,L))}{a_0} + \frac{L}{a_1} \right)$$
as the elastic energy in the body at time $t$. Therefore, since 
$$\E(t) = \frac{\sigma(t)^2}{2} \left( \frac{l(t)}{a_0} + \frac{L}{a_1} \right) + \kappa l(t),$$
we can write the following energy balance  
$$
\mathcal Q(t) +  \int_0^t \frac{ d }{ ds } \left(  \frac{\sigma(s) p(s)(\{ 0,L \})}{2} + \kappa  l(s)\right)  \, ds = \mathcal Q(0) +  \int_0^t  \sigma(s) \big[ \dot w(s) \big]^L_0 \, ds
$$
where the left hand side is the sum of the elastic energy and a dissipative cost due to damage. 
\item As explained above, the Griffith type Evolution Law states that damage can only grow when the stress $\sigma$ saturates the constraint. This threshold condition generalizes the Initial Constituve Law \eqref{eq:Constitutive Equation Initial Time} to the quasi-static setting. As will be explained in the Section \ref{section:Energy Balance}, this Constitutive Law \eqref{eq:Constitutive Equation Initial Time} a priori does not propagate to subsequent times through the evolution process, unless the prescribed boundary condition satisfies \eqref{eq:CNS w}, corresponding to the case of perfect plasticity. In this case, one can check that $\big( 2 \kappa a_0 - \sigma(t)^2 \big) l(t) = 0$ for all $t \in [0,T]$.
\end{itemize}
\end{rem}
Motivated by the static analysis led in \cite{BIR} where the authors have shown how brittle damage can lead to Hencky perfect plasticity, it is natural to extend this analysis to quasi-static evolutions and inquire whether $(u,e,p,\sigma)$ is of perfect plasticity type or not. Following \cite[Definition 4.2]{DMDSM}, in order for this quasi-static evolution to be of perfect plasticity type, it only remains to prove the Energy Balance:
\begin{equation}\label{eq:Energy Balance}
\frac{La_1}{2} e(t)^2 + \sqrt{2 \kappa a_0} \, \int_0^t \left\lvert \dot p(s) \right\rvert([0,L]) \, ds  = \frac{La_1}{2}  e(0)^2  + \int_0^t \int_0^L \sigma(s)    \dot w'(s)(x) \, dx ds  
\end{equation}
for all time $t \in [0,T]$, where we used that $\int_0^t \left\lvert \dot p(s) \right\rvert([0,L]) \, ds = \mathcal V(p;0,t)$ according to \cite[Theorem 7.1]{DMDSM}. As $u$, $\sigma$, $p$, $\mu$ and $l$ are all absolutely continuous on $[0,T]$, we have the following Proposition:
\begin{prop}
For all $t \in [0,T]$,
\begin{equation}\label{eq:Fake Energy Balance}
\frac{La_1}{2} e(t)^2 + \int_0^t \sigma(s) \dot p(s)([0,L]) \, ds = \frac{L a_1}{2} e(0)^2 + \int_0^t \sigma(s) \big[ \dot w(s) \big]^L_0 \, ds.
\end{equation}
\end{prop} 
\begin{proof}
Indeed, for all $t \in [0,T]$, \eqref{eq:E(t)} and \eqref{eq:2} entail that
\begin{multline*}
\E(t) = \frac{La_1}{2} e(t)^2 + l(t) \left( \frac{\sigma(t)^2}{2a_0} + \kappa \right) \\
 =  \frac{La_1}{2} e(t)^2 + l(0) \left( \frac{\sigma(0)^2}{2a_0} + \kappa \right) + \int_0^t \frac{\rm d}{\rm{ds}} \left( l \left( \frac{\sigma^2}{2a_0} + \kappa \right) \right)(s) \, ds.
\end{multline*}
Using that $ \dot{ \wideparen{ l \sigma^2  }} = \dot l  \sigma^2 + 2 l \sigma \dot \sigma $ together with \eqref{eq:Constitutive Equation Initial Time} and Theorem \ref{thm:DAMAGE}, we infer that
$$
\E(t) = \frac{La_1}{2} e(t)^2  + \sqrt{2 \kappa a_0} \left\lvert \sigma(0) \right\rvert \frac{l(0)}{a_0}  + \int_0^t  \left( \frac{\dot l}{a_0} \sqrt{2 \kappa a_0} \left\lvert \sigma \right\rvert + \frac{l \sigma \dot \sigma}{a_0} \right)  \, ds.
$$
Thus \eqref{eq:link mu l}, Theorem \ref{thm:DAMAGE} and \eqref{eq:dot mu dot l} entail that 
\begin{multline*}
\E(t) = \frac{La_1}{2} e(t)^2  + \sqrt{2 \kappa a_0} \left\lvert p(0) \right\rvert ([0,L])  + \int_0^t  \sigma \left( \frac{\dot l \sigma + l  \dot \sigma}{a_0} \right)  \, ds \\
 = \frac{La_1}{2} e(t)^2  + \sqrt{2 \kappa a_0} \left\lvert p(0) \right\rvert ([0,L])  + \int_0^t  \sigma(s) \dot p(s) ([0,L])  \, ds.
\end{multline*}
Then, \eqref{eq:E(t)}, \eqref{eq:Initial Time} and \eqref{eq:Constitutive Equation Initial Time} complete the proof of \eqref{eq:Fake Energy Balance}.
\end{proof}
Therefore, due to the homogeneity in space of $\sigma$ together with the stress constraint, \eqref{eq:Fake Energy Balance} ensures that the upper bound inequality of \eqref{eq:Energy Balance} is always satisfied:
\begin{prop}\label{cor:Upper Bound} For all $t \in [0,T]$,
$$\frac{La_1}{2} e(t)^2 + \sqrt{2 \kappa a_0} \, \int_0^t \left\lvert \dot p(s) \right\rvert([0,L]) \, ds   \geq  \frac{La_1}{2}  e(0)^2  + \int_0^t \int_0^L \sigma(s)    \dot w'(s)(x) \, dx ds. $$ 
\end{prop}
Having all the previous results in mind, one would naturally be tempted to intuit the validity of the Energy Balance \eqref{eq:Energy Balance}, hence proving that the quasi-static damage evolution $(u,e,p,\sigma)$ is indeed one of perfect plasticity. Surprisingly, the interplay between relaxation and irreversibility of the damage is not stable through time evolutions. Indeed, depending on the choice of the prescribed Dirichlet boundary condition $w \in AC([0,T];H^1(\R))$, the effective quasi-static damage evolution may not be of perfect plasticity type, as illustrated in the example of Figure \ref{fig:3}. Understanding on which condition the effective quasi-static evolution is of perfect plasticity type is the content of the next section.

\section{Energy Balance}\label{section:Energy Balance}
We can prove that the Energy Balance \eqref{eq:Energy Balance} is satisfied if and only if $\sigma$ saturates the constraint once $l$ is non-zero, until the end of the process (see Figure \ref{fig:2}).
\begin{lem}
Let ${\displaystyle t_0 = \sup \, \{ t \in [0,T] \, : \, \mu(t) \big( [0,L] \big) = 0 \} }$. The Energy Balance \eqref{eq:Energy Balance} is satisfied if and only if $\sigma$ saturates the constraint during the whole time interval $[t_0,T]$ , {\it i.e.}
\begin{equation}\label{eq:Naive CNS}
\left\lvert \sigma(t) \right\rvert = \sqrt{2 \kappa a_0} \quad \text{for all } t \in [t_0,T].
\end{equation}
\end{lem}
\begin{proof}
Let us heuristically explain the argument. On the one hand, let us assume that \eqref{eq:Naive CNS} holds. Since $\mu(t) = 0 = p(t)$ in $\M \big( [0,L] \big)$ for all $t \in [0,t_0]$ and $\left\lvert \sigma(t) \right\rvert = \sqrt{2 \kappa a_0}$ for all $t \in [t_0,T]$, the Flow-Rule holds:
\begin{equation}\label{eq:FlowRuleBIS}
\sigma(t) \dot p(t) \big( [0,L] \big) = \sqrt{2 \kappa a_0} \left\lvert \dot p(t) \right\rvert \big( [0,L] \big) \quad \text{for } \LL^1 \text{-a.e. } t \in [0,T].
\end{equation}
This is immediate during the time interval $[0,t_0]$, while during the time interval $[t_0,T]$, by continuity we infer that $\sigma$ is constant and either $\sigma \equiv \sqrt{2 \kappa a_0}$ or $\sigma \equiv - \sqrt{2 \kappa a_0}$. Especially, since $p = \frac{\sigma}{a_0} \mu$ and $\mu$ is non-decreasing in time, we infer that $ p $ is either non-decreasing or non-increasing in time on $[t_0,T]$, according to the sign of $\sigma$. Hence, $\sigma \dot p = \frac{\sigma^2}{a_0} \dot \mu = \sqrt{2 \kappa a_0} \left\lvert \dot p \right\rvert$ in $\M \big( [0,L]; \R^+ \big)$. Therefore, the validity of the Energy Balance \eqref{eq:Energy Balance} follows from \eqref{eq:Fake Energy Balance} and \eqref{eq:FlowRuleBIS}.

On the other hand, \eqref{eq:Naive CNS} is also a necessary condition. Let us assume that the Energy Balance \eqref{eq:Energy Balance} is satisfied. From \eqref{eq:Fake Energy Balance}, we infer that 
$$
\int_0^t \sigma(s) \dot p(s)\big([0,L] \big) \, ds = \sqrt{2 \kappa a_0} \int_0^t \left\lvert \dot p(s) \right\rvert \big([0,L] \big) \, ds
$$ 
for all $t \in [0,T]$. As $ \sigma \dot p\big([0,L] \big) \leq \sqrt{2 \kappa a_0}  \left\lvert \dot p \right\rvert \big([0,L] \big)$ always holds, we deduce that equality \eqref{eq:FlowRuleBIS} is satisfied $\LL^1$-a.e. on $[0,T]$. In particular, \eqref{eq:Naive CNS} must hold, otherwise the Flow-Rule will not be satisfied during a non $\LL^1$-negligible set of times in $[t_0,T]$. Indeed, if $\left\lvert \sigma(t) \right\rvert < \sqrt{2 \kappa a_0}$ for some $t \in (t_0,T)$, considering the maximal time interval $t \in I \subset (t_0,T]$ during which $\sigma$ never saturates the constraint, we get by continuity of $\sigma$ and Theorem \ref{thm:DAMAGE} that ${\rm Int} (I)$ is a non empty interval and $\mu$ is a constant non-zero measure on $I$. Moreover, there exists $E \subset I$ such that $\LL^1(E) >0$ and $\dot \sigma \neq 0$ on $E$. If such was not the case, $\sigma$ would be constant on $I$, which is impossible by maximality of the interval. Thus, one simultaneously has $\dot p \big( [0,L] \big) = \dot \sigma / a_0 l \neq 0$ and $\left\lvert \sigma \right\rvert < \sqrt{2 \kappa a_0}$ on $E$, so that $\left\lvert \sigma \right\rvert \left\lvert \dot \sigma /a_0 \right\rvert   l < \sqrt{2 \kappa a_0}  \left\lvert \dot \sigma / a_0 \right\rvert l$ which is in contradiction with the Flow-Rule \eqref{eq:FlowRuleBIS}.
\end{proof}
Yet, this sufficient and necessary condition \eqref{eq:Naive CNS} relies on the definition of the time $t_0$. It remains to find an equivalent condition which can be expressed only in terms of the data of the setting. This is the content of Theorem \ref{thm:CNS FR}, illustrated in Figure \ref{fig:1}.

\begin{figure}[hbtp]
\begin{tikzpicture}[line cap=round,line join=round,>=triangle 45,x=0.8cm,y=0.8cm]
\clip(-3,-1.5) rectangle (14.,7.);
\draw [line width=1.pt] (0.,7.)-- (0.,-1.);
\draw [line width=1.pt] (-1.5,0.)-- (12.5,0.);
\draw [line width=1.pt] (0.,7.)-- (-0.1,6.8);
\draw [line width=1.pt] (0.,7.)-- (0.1,6.8);
\draw [line width=1.pt] (12.5,0.)-- (12.3,0.1);
\draw [line width=1.pt] (12.5,0.)-- (12.3,-0.1);
\draw [line width=1.pt,dash pattern=on 4pt off 4pt] (-0.3,3.5)-- (11.,3.5);
\draw [line width=1.pt,dash pattern=on 4pt off 4pt] (11.,3.5)-- (11.,-0.2);
\draw (10.7,-0.2) node[anchor=north west] {$\mathlarger{\mathlarger{T}}$};
\draw (0.,-0.2) node[anchor=north west] {$\mathlarger{\mathlarger{0}}$};
\draw (-2.6,4.5) node[anchor=north west] {$\mathlarger{\mathlarger{\sqrt{ 2 \kappa a_0} \frac{L}{a_1}}}$};
\draw (0.3,7.2) node[anchor=north west] {$\mathlarger{\mathlarger{\left\lvert \big[ w(t) \big]^L_0 \right\rvert}}$};
\draw (1.2,-0.1) node[anchor=north west] {$\mathlarger{\mathlarger{\inf \left\{ t \in [0,T]\, : \, \left\lvert \big[ w(t) \big]^L_0 \right\rvert > \sqrt{2 \kappa a_0}\frac{L}{a_1} \right\}}}$};
\draw (12.5,0.4) node[anchor=north west] {\large{time}};
\draw [line width=1.pt,dash pattern=on 4pt off 4pt] (4.4,-0.2)-- (4.4,3.5);
\draw[line width=1.2pt] (0.01682648858733636,1.8702242712598633) -- (0.0715858523629325,1.749767784091655);
\draw[line width=1.2pt] (0.0715858523629325,1.749767784091655) -- (0.12634521613852862,1.5876718976017739);
\draw[line width=1.2pt] (0.12634521613852862,1.5876718976017739) -- (0.18110457991412474,1.4587015300242658);
\draw[line width=1.2pt] (0.18110457991412474,1.4587015300242658) -- (0.23586394368972086,1.3964001387000735);
\draw[line width=1.2pt] (0.23586394368972086,1.3964001387000735) -- (0.290623307465317,1.4076016423278617);
\draw[line width=1.2pt] (0.290623307465317,1.4076016423278617) -- (0.34538267124091315,1.483083134310561);
\draw[line width=1.2pt] (0.34538267124091315,1.483083134310561) -- (0.4001420350165093,1.6051988627663154);
\draw[line width=1.2pt] (0.4001420350165093,1.6051988627663154) -- (0.45490139879210545,1.7531785224580163);
\draw[line width=1.2pt] (0.45490139879210545,1.7531785224580163) -- (0.5096607625677015,1.9066416881331145);
\draw[line width=1.2pt] (0.5096607625677015,1.9066416881331145) -- (0.5644201263432976,2.047771244642863);
\draw[line width=1.2pt] (0.5644201263432976,2.047771244642863) -- (0.6191794901188937,2.1624985214094243);
\draw[line width=1.2pt] (0.6191794901188937,2.1624985214094243) -- (0.6739388538944898,2.2409785861785525);
\draw[line width=1.2pt] (0.6739388538944898,2.2409785861785525) -- (0.7286982176700859,2.277573284286448);
\draw[line width=1.2pt] (0.7286982176700859,2.277573284286448) -- (0.783457581445682,2.2705099773463604);
\draw[line width=1.2pt] (0.783457581445682,2.2705099773463604) -- (0.8382169452212781,2.221343706414017);
\draw[line width=1.2pt] (0.8382169452212781,2.221343706414017) -- (0.8929763089968742,2.1343181181997677);
\draw[line width=1.2pt] (0.8929763089968742,2.1343181181997677) -- (0.9477356727724703,2.015694622078895);
\draw[line width=1.2pt] (0.9477356727724703,2.015694622078895) -- (1.0024950365480665,1.8730987657454485);
\draw[line width=1.2pt] (1.0024950365480665,1.8730987657454485) -- (1.0572544003236626,1.7149167772163638);
\draw[line width=1.2pt] (1.0572544003236626,1.7149167772163638) -- (1.1120137640992587,1.54976281841558);
\draw[line width=1.2pt] (1.1120137640992587,1.54976281841558) -- (1.1667731278748548,1.3860280563473266);
\draw[line width=1.2pt] (1.1667731278748548,1.3860280563473266) -- (1.2215324916504509,1.2315156167067767);
\draw[line width=1.2pt] (1.2215324916504509,1.2315156167067767) -- (1.276291855426047,1.0931603696795331);
\draw[line width=1.2pt] (1.276291855426047,1.0931603696795331) -- (1.331051219201643,0.9768289160194658);
\draw[line width=1.2pt] (1.331051219201643,0.9768289160194658) -- (1.3858105829772391,0.8871927680629821);
\draw[line width=1.2pt] (1.3858105829772391,0.8871927680629821) -- (1.4405699467528352,0.827666287068004);
\draw[line width=1.2pt] (1.4405699467528352,0.827666287068004) -- (1.4953293105284313,0.8004002253725491);
\draw[line width=1.2pt] (1.4953293105284313,0.8004002253725491) -- (1.5500886743040274,0.8063215502476058);
\draw[line width=1.2pt] (1.5500886743040274,0.8063215502476058) -- (1.6048480380796235,0.845210451034949);
\draw[line width=1.2pt] (1.6048480380796235,0.845210451034949) -- (1.6596074018552196,0.9158059358682411);
\draw[line width=1.2pt] (1.6596074018552196,0.9158059358682411) -- (1.7143667656308157,1.0159321164654658);
\draw[line width=1.2pt] (1.7143667656308157,1.0159321164654658) -- (1.7691261294064118,1.142638086425631);
\draw[line width=1.2pt] (1.7691261294064118,1.142638086425631) -- (1.8238854931820079,1.2923451637764114);
\draw[line width=1.2pt] (1.8238854931820079,1.2923451637764114) -- (1.878644856957604,1.4609961492397117);
\draw[line width=1.2pt] (1.878644856957604,1.4609961492397117) -- (1.9334042207332,1.6442021158067428);
\draw[line width=1.2pt] (1.9334042207332,1.6442021158067428) -- (1.9881635845087962,1.8373830696257958);
\draw[line width=1.2pt] (1.9881635845087962,1.8373830696257958) -- (2.0429229482843922,2.0358995909306437);
\draw[line width=1.2pt] (2.0429229482843922,2.0358995909306437) -- (2.0976823120599883,2.2351732664890442);
\draw[line width=1.2pt] (2.0976823120599883,2.2351732664890442) -- (2.1524416758355844,2.4307943560181364);
\draw[line width=1.2pt] (2.1524416758355844,2.4307943560181364) -- (2.2072010396111805,2.618615691856517);
\draw[line width=1.2pt] (2.2072010396111805,2.618615691856517) -- (2.2619604033867766,2.794832294204509);
\draw[line width=1.2pt] (2.2619604033867766,2.794832294204509) -- (2.3167197671623727,2.956046595712489);
\draw[line width=1.2pt] (2.3167197671623727,2.956046595712489) -- (2.371479130937969,3.09931951279322);
\draw[line width=1.2pt] (2.371479130937969,3.09931951279322) -- (2.426238494713565,3.2222078814069244);
\draw[line width=1.2pt] (2.426238494713565,3.2222078814069244) -- (2.480997858489161,3.3227889974781775);
\draw[line width=1.2pt] (2.480997858489161,3.3227889974781775) -- (2.535757222264757,3.3996731721432707);
\draw[line width=1.2pt] (2.535757222264757,3.3996731721432707) -- (2.590516586040353,3.452005335398611);
\draw[line width=1.2pt] (2.590516586040353,3.452005335398611) -- (2.6452759498159493,3.4794568040720195);
\draw[line width=1.2pt] (2.6452759498159493,3.4794568040720195) -- (2.7000353135915454,3.4822083768350893);
\draw[line width=1.2pt] (2.7000353135915454,3.4822083768350893) -- (2.7547946773671415,3.460925935409368);
\draw[line width=1.2pt] (2.7547946773671415,3.460925935409368) -- (2.8095540411427375,3.4167297220505195);
\draw[line width=1.2pt] (2.8095540411427375,3.4167297220505195) -- (2.8643134049183336,3.3511584333064532);
\draw[line width=1.2pt] (2.8643134049183336,3.3511584333064532) -- (2.9190727686939297,3.266129223024214);
\draw[line width=1.2pt] (2.9190727686939297,3.266129223024214) -- (2.973832132469526,3.1638946473080143);
\draw[line width=1.2pt] (2.973832132469526,3.1638946473080143) -- (3.028591496245122,3.0469975138868826);
\draw[line width=1.2pt] (3.028591496245122,3.0469975138868826) -- (3.083350860020718,2.918224521024152);
\draw[line width=1.2pt] (3.083350860020718,2.918224521024152) -- (3.138110223796314,2.780559489208093);
\draw[line width=1.2pt] (3.138110223796314,2.780559489208093) -- (3.19286958757191,2.6371369045673436);
\draw[line width=1.2pt] (3.19286958757191,2.6371369045673436) -- (3.2476289513475063,2.4911964080928586);
\draw[line width=1.2pt] (3.2476289513475063,2.4911964080928586) -- (3.3023883151231024,2.3460387808551926);
\draw[line width=1.2pt] (3.3023883151231024,2.3460387808551926) -- (3.3571476788986985,2.204983893742678);
\draw[line width=1.2pt] (3.3571476788986985,2.204983893742678) -- (3.4119070426742946,2.071331011825699);
\draw[line width=1.2pt] (3.4119070426742946,2.071331011825699) -- (3.4666664064498907,1.9483217690665484);
\draw[line width=1.2pt] (3.4666664064498907,1.9483217690665484) -- (3.5214257702254868,1.839106059339333);
\draw[line width=1.2pt] (3.5214257702254868,1.839106059339333) -- (3.576185134001083,1.7467110250241455);
\draw[line width=1.2pt] (3.576185134001083,1.7467110250241455) -- (3.630944497776679,1.6740132650694035);
\draw[line width=1.2pt] (3.630944497776679,1.6740132650694035) -- (3.685703861552275,1.6237143305228439);
\draw[line width=1.2pt] (3.685703861552275,1.6237143305228439) -- (3.740463225327871,1.5983195271536248);
\draw[line width=1.2pt] (3.740463225327871,1.5983195271536248) -- (3.795222589103467,1.600120001872828);
\draw[line width=1.2pt] (3.795222589103467,1.600120001872828) -- (3.8499819528790633,1.6311780520800236);
\draw[line width=1.2pt] (3.8499819528790633,1.6311780520800236) -- (3.9047413166546594,1.6933155646316818);
\draw[line width=1.2pt] (3.9047413166546594,1.6933155646316818) -- (3.9595006804302555,1.7881054636076386);
\draw[line width=1.2pt] (3.9595006804302555,1.7881054636076386) -- (4.014260044205852,1.9168660231728427);
\draw[line width=1.2pt] (4.014260044205852,1.9168660231728427) -- (4.069019407981448,2.0806578832951472);
\draw[line width=1.2pt] (4.069019407981448,2.0806578832951472) -- (4.123778771757044,2.2802835915701727);
\draw[line width=1.2pt] (4.123778771757044,2.2802835915701727) -- (4.17853813553264,2.516289483595144);
\draw[line width=1.2pt] (4.17853813553264,2.516289483595144) -- (4.233297499308236,2.7889697068950214);
\draw[line width=1.2pt] (4.233297499308236,2.7889697068950214) -- (4.288056863083832,3.09837218900741);
\draw[line width=1.2pt] (4.288056863083832,3.09837218900741) -- (4.342816226859428,3.444306348654671);
\draw[line width=1.2pt] (4.342816226859428,3.462065095371185) -- (4.397575590635024,3.629948457889982);
\draw[line width=1.2pt] (4.397575590635024,3.629948457889982) -- (4.45233495441062,3.7575228677637478);
\draw[line width=1.2pt] (4.45233495441062,3.7575228677637478) -- (4.507094318186216,3.8499689773111587);
\draw[line width=1.2pt] (4.507094318186216,3.8499689773111587) -- (4.5618536819618125,3.9185796816506286);
\draw[line width=1.2pt] (4.5618536819618125,3.9185796816506286) -- (4.616613045737409,3.971099837521969);
\draw[line width=1.2pt] (4.616613045737409,3.971099837521969) -- (4.671372409513005,4.01244228136444);
\draw[line width=1.2pt] (4.671372409513005,4.01244228136444) -- (4.726131773288601,4.04576686444668);
\draw[line width=1.2pt] (4.726131773288601,4.04576686444668) -- (4.780891137064197,4.073168757540306);
\draw[line width=1.2pt] (4.780891137064197,4.073168757540306) -- (4.835650500839793,4.096081927185128);
\draw[line width=1.2pt] (4.835650500839793,4.096081927185128) -- (4.890409864615389,4.115516903342197);
\draw[line width=1.2pt] (4.890409864615389,4.115516903342197) -- (4.945169228390985,4.132204577580305);
\draw[line width=1.2pt] (4.945169228390985,4.132204577580305) -- (4.999928592166581,4.146685836904015);
\draw[line width=1.2pt] (4.999928592166581,4.146685836904015) -- (5.054687955942177,4.159369112625393);
\draw[line width=1.2pt] (5.054687955942177,4.159369112625393) -- (5.109447319717773,4.170568357710904);
\draw[line width=1.2pt] (5.109447319717773,4.170568357710904) -- (5.1642066834933695,4.180528742486414);
\draw[line width=1.2pt] (5.1642066834933695,4.180528742486414) -- (5.218966047268966,4.189444437320316);
\draw[line width=1.2pt] (5.218966047268966,4.189444437320316) -- (5.273725411044562,4.1974711722288705);
\draw[line width=1.2pt] (5.273725411044562,4.1974711722288705) -- (5.328484774820158,4.204735272146396);
\draw[line width=1.2pt] (5.328484774820158,4.204735272146396) -- (5.383244138595754,4.211340265956178);
\draw[line width=1.2pt] (5.383244138595754,4.211340265956178) -- (5.43800350237135,4.217371794464584);
\draw[line width=1.2pt] (5.43800350237135,4.217371794464584) -- (5.492762866146946,4.222901305715212);
\draw[line width=1.2pt] (5.492762866146946,4.222901305715212) -- (5.547522229922542,4.227988872539378);
\draw[line width=1.2pt] (5.547522229922542,4.227988872539378) -- (5.602281593698138,4.2326853658086705);
\draw[line width=1.2pt] (5.602281593698138,4.2326853658086705) -- (5.657040957473734,4.237034172975937);
\draw[line width=1.2pt] (5.657040957473734,4.237034172975937) -- (5.7118003212493305,4.241202860477588);
\draw[line width=1.2pt] (5.7118003212493305,4.241202860477588) -- (5.7665596850249266,4.2473014430365765);
\draw[line width=1.2pt] (5.7665596850249266,4.2473014430365765) -- (5.821319048800523,4.259248256590282);
\draw[line width=1.2pt] (5.821319048800523,4.259248256590282) -- (5.876078412576119,4.278349078426568);
\draw[line width=1.2pt] (5.876078412576119,4.278349078426568) -- (5.930837776351715,4.303374808074773);
\draw[line width=1.2pt] (5.930837776351715,4.303374808074773) -- (5.985597140127311,4.332359631606364);
\draw[line width=1.2pt] (5.985597140127311,4.332359631606364) -- (6.040356503902907,4.363534196995246);
\draw[line width=1.2pt] (6.040356503902907,4.363534196995246) -- (6.095115867678503,4.39557039654044);
\draw[line width=1.2pt] (6.095115867678503,4.39557039654044) -- (6.149875231454099,4.4275554535922526);
\draw[line width=1.2pt] (6.149875231454099,4.4275554535922526) -- (6.204634595229695,4.4588975608807475);
\draw[line width=1.2pt] (6.204634595229695,4.4588975608807475) -- (6.259393959005291,4.489233686031703);
\draw[line width=1.2pt] (6.259393959005291,4.489233686031703) -- (6.3141533227808875,4.518357295259515);
\draw[line width=1.2pt] (6.3141533227808875,4.518357295259515) -- (6.368912686556484,4.546166234576058);
\draw[line width=1.2pt] (6.368912686556484,4.546166234576058) -- (6.42367205033208,4.572626572400076);
\draw[line width=1.2pt] (6.42367205033208,4.572626572400076) -- (6.478431414107676,4.597748006092006);
\draw[line width=1.2pt] (6.478431414107676,4.597748006092006) -- (6.533190777883272,4.621567320457943);
\draw[line width=1.2pt] (6.533190777883272,4.621567320457943) -- (6.587950141658868,4.644137347951213);
\draw[line width=1.2pt] (6.587950141658868,4.644137347951213) -- (6.642709505434464,4.665519654598475);
\draw[line width=1.2pt] (6.642709505434464,4.665519654598475) -- (6.69746886921006,4.685779738795178);
\draw[line width=1.2pt] (6.69746886921006,4.685779738795178) -- (6.752228232985656,4.704983921810693);
\draw[line width=1.2pt] (6.752228232985656,4.704983921810693) -- (6.806987596761252,4.723197375703676);
\draw[line width=1.2pt] (6.806987596761252,4.723197375703676) -- (6.861746960536848,4.740482914508512);
\draw[line width=1.2pt] (6.861746960536848,4.740482914508512) -- (6.9165063243124445,4.756900295824163);
\draw[line width=1.2pt] (6.9165063243124445,4.756900295824163) -- (6.971265688088041,4.7725058616031415);
\draw[line width=1.2pt] (6.971265688088041,4.7725058616031415) -- (7.026025051863637,4.787352402058739);
\draw[line width=1.2pt] (7.026025051863637,4.787352402058739) -- (7.080784415639233,4.801489163928769);
\draw[line width=1.2pt] (7.080784415639233,4.801489163928769) -- (7.135543779414829,4.814961949688291);
\draw[line width=1.2pt] (7.135543779414829,4.814961949688291) -- (7.190303143190425,4.827813271584722);
\draw[line width=1.2pt] (7.190303143190425,4.827813271584722) -- (7.245062506966021,4.840082536179831);
\draw[line width=1.2pt] (7.245062506966021,4.840082536179831) -- (7.299821870741617,4.851806243171764);
\draw[line width=1.2pt] (7.299821870741617,4.851806243171764) -- (7.354581234517213,4.863018187815317);
\draw[line width=1.2pt] (7.354581234517213,4.863018187815317) -- (7.409340598292809,4.873749660058511);
\draw[line width=1.2pt] (7.409340598292809,4.873749660058511) -- (7.464099962068405,4.884029636112114);
\draw[line width=1.2pt] (7.464099962068405,4.884029636112114) -- (7.5188593258440015,4.8938849599380845);
\draw[line width=1.2pt] (7.5188593258440015,4.8938849599380845) -- (7.573618689619598,4.903340513338856);
\draw[line width=1.2pt] (7.573618689619598,4.903340513338856) -- (7.628378053395194,4.912419374128349);
\draw[line width=1.2pt] (7.628378053395194,4.912419374128349) -- (7.68313741717079,4.921142962390009);
\draw[line width=1.2pt] (7.68313741717079,4.921142962390009) -- (7.737896780946386,4.929531175161838);
\draw[line width=1.2pt] (7.737896780946386,4.929531175161838) -- (7.792656144721982,4.937602510092647);
\draw[line width=1.2pt] (7.792656144721982,4.937602510092647) -- (7.847415508497578,4.945374178728852);
\draw[line width=1.2pt] (7.847415508497578,4.945374178728852) -- (7.902174872273174,4.952862210145813);
\draw[line width=1.2pt] (7.902174872273174,4.952862210145813) -- (7.95693423604877,4.960081569202656);
\draw[line width=1.2pt] (7.95693423604877,4.960081569202656) -- (8.011693599824367,4.967175465477924);
\draw[line width=1.2pt] (8.011693599824367,4.967175465477924) -- (8.066452963599964,4.976228698019497);
\draw[line width=1.2pt] (8.066452963599964,4.976228698019497) -- (8.121212327375561,4.991145342246115);
\draw[line width=1.2pt] (8.121212327375561,4.991145342246115) -- (8.175971691151158,5.0132248190287125);
\draw[line width=1.2pt] (8.175971691151158,5.0132248190287125) -- (8.230731054926755,5.041230493396029);
\draw[line width=1.2pt] (8.230731054926755,5.041230493396029) -- (8.285490418702352,5.073188114838083);
\draw[line width=1.2pt] (8.285490418702352,5.073188114838083) -- (8.34024978247795,5.107320484690813);
\draw[line width=1.2pt] (8.34024978247795,5.107320484690813) -- (8.395009146253546,5.14229283800916);
\draw[line width=1.2pt] (8.395009146253546,5.14229283800916) -- (8.449768510029143,5.177186982241173);
\draw[line width=1.2pt] (8.449768510029143,5.177186982241173) -- (8.50452787380474,5.211406794137539);
\draw[line width=1.2pt] (8.50452787380474,5.211406794137539) -- (8.559287237580337,5.2445858459420345);
\draw[line width=1.2pt] (8.559287237580337,5.2445858459420345) -- (8.614046601355934,5.276514964145197);
\draw[line width=1.2pt] (8.614046601355934,5.276514964145197) -- (8.668805965131531,5.307089972509306);
\draw[line width=1.2pt] (8.668805965131531,5.307089972509306) -- (8.723565328907128,5.336275422438853);
\draw[line width=1.2pt] (8.723565328907128,5.336275422438853) -- (8.778324692682725,5.364079909020205);
\draw[line width=1.2pt] (8.778324692682725,5.364079909020205) -- (8.833084056458322,5.390539456396056);
\draw[line width=1.2pt] (8.833084056458322,5.390539456396056) -- (8.887843420233919,5.415706418516413);
\draw[line width=1.2pt] (8.887843420233919,5.415706418516413) -- (8.942602784009516,5.4396421163951105);
\draw[line width=1.2pt] (8.942602784009516,5.4396421163951105) -- (8.997362147785113,5.462411996789502);
\draw[line width=1.2pt] (8.997362147785113,5.462411996789502) -- (9.05212151156071,5.484082489437277);
\draw[line width=1.2pt] (9.05212151156071,5.484082489437277) -- (9.106880875336307,5.504719007236181);
\draw[line width=1.2pt] (9.106880875336307,5.504719007236181) -- (9.161640239111904,5.5243847142046345);
\draw[line width=1.2pt] (9.161640239111904,5.5243847142046345) -- (9.2163996028875,5.5431398075501495);
\draw[line width=1.2pt] (9.2163996028875,5.5431398075501495) -- (9.271158966663098,5.561041142003718);
\draw[line width=1.2pt] (9.271158966663098,5.561041142003718) -- (9.325918330438695,5.578142079824557);
\draw[line width=1.2pt] (9.325918330438695,5.578142079824557) -- (9.380677694214292,5.594492487297);
\draw[line width=1.2pt] (9.380677694214292,5.594492487297) -- (9.435437057989889,5.610138823971508);
\draw[line width=1.2pt] (9.435437057989889,5.610138823971508) -- (9.490196421765486,5.625124288242953);
\draw[line width=1.2pt] (9.490196421765486,5.625124288242953) -- (9.544955785541083,5.639488994718269);
\draw[line width=1.2pt] (9.544955785541083,5.639488994718269) -- (9.59971514931668,5.653270166953137);
\draw[line width=1.2pt] (9.59971514931668,5.653270166953137) -- (9.654474513092277,5.6665023347137975);
\draw[line width=1.2pt] (9.654474513092277,5.6665023347137975) -- (9.709233876867874,5.679217528745772);
\draw[line width=1.2pt] (9.709233876867874,5.679217528745772) -- (9.76399324064347,5.691445468651073);
\draw[line width=1.2pt] (9.76399324064347,5.691445468651073) -- (9.818752604419068,5.703213741262471);
\draw[line width=1.2pt] (9.818752604419068,5.703213741262471) -- (9.873511968194665,5.7145479681139575);
\draw[line width=1.2pt] (9.873511968194665,5.7145479681139575) -- (9.928271331970262,5.725471961417847);
\draw[line width=1.2pt] (9.928271331970262,5.725471961417847) -- (9.983030695745859,5.736007868493598);
\draw[line width=1.2pt] (9.983030695745859,5.736007868493598) -- (10.037790059521456,5.746176304936904);
\draw[line width=1.2pt] (10.037790059521456,5.746176304936904) -- (10.092549423297053,5.755996477029132);
\draw[line width=1.2pt] (10.092549423297053,5.755996477029132) -- (10.14730878707265,5.765486294008575);
\draw[line width=1.2pt] (10.14730878707265,5.765486294008575) -- (10.202068150848246,5.77466247088493);
\draw[line width=1.2pt] (10.202068150848246,5.77466247088493) -- (10.256827514623843,5.783540622497565);
\draw[line width=1.2pt] (10.256827514623843,5.783540622497565) -- (10.31158687839944,5.7921353495101755);
\draw[line width=1.2pt] (10.31158687839944,5.7921353495101755) -- (10.366346242175037,5.800460317009701);
\draw[line width=1.2pt] (10.366346242175037,5.800460317009701) -- (10.421105605950634,5.808528326342214);
\draw[line width=1.2pt] (10.421105605950634,5.808528326342214) -- (10.475864969726231,5.816351380777846);
\draw[line width=1.2pt] (10.475864969726231,5.816351380777846) -- (10.530624333501828,5.823940745553806);
\draw[line width=1.2pt] (10.530624333501828,5.823940745553806) -- (10.585383697277425,5.831307002801165);
\draw[line width=1.2pt] (10.585383697277425,5.831307002801165) -- (10.640143061053022,5.838460101818814);
\draw[line width=1.2pt] (10.640143061053022,5.838460101818814) -- (10.69490242482862,5.845409405117671);
\draw[line width=1.2pt] (10.69490242482862,5.845409405117671) -- (10.749661788604216,5.852163730620246);
\draw[line width=1.2pt] (10.749661788604216,5.852163730620246) -- (10.804421152379813,5.858731390365412);
\draw[line width=1.2pt] (10.804421152379813,5.858731390365412) -- (10.85918051615541,5.865120226035679);
\draw[line width=1.2pt] (10.85918051615541,5.865120226035679) -- (10.913939879931007,5.871337641594453);
\draw[line width=1.2pt] (10.913939879931007,5.871337641594453) -- (10.968699243706604,5.877390633293524);
\draw[line width=1.2pt] (10.968699243706604,5.877390633293524) -- (11.023458607482201,5.88328581728633);
\draw (5.,2.2) node[anchor=north west] {$\mathlarger{\mathlarger{\left\lvert \big[ w \big]^L_0 \right\rvert \text{ non-decreasing}}}$};
\end{tikzpicture}
\caption{}
\label{fig:1}
\end{figure}

\begin{proof}[Proof of Theorem \ref{thm:CNS FR}:]
{\it Necessary condition.} Assume that the Energy Balance \eqref{eq:Energy Balance} is satisfied at all time $t \in [0,T]$. Then, Theorem 5.2 and Theorem 6.1 of \cite{DMDSM} ensure that $(u,e,p,\sigma) \in AC \big( [0,T];BV((0,L)) \times \R \times \M([0,L]) \times \KK \big)$ satisfies the Flow Rule:
\begin{equation}\label{eq:FR}
 \sigma(t) \dot p(t) \big( [0,L] \big) = \sqrt{2\kappa a_0} \left\lvert \dot p(t) \right\rvert \big( [0,L] \big) \quad \text{for } \LL^1 \text{-a.e. } t \in [0,T] .
\end{equation} 
On the one hand, \eqref{eq:2} together with the trace theorem in $BV((0,L))$ entail that, at all time, $Du ((0,L)) = \big[ u \big]^L_0 = L e + p((0,L)) = \big[ w \big]^L_0 - \frac{\sigma l}{a_0} + p((0,L))$ so that
$$
p([0,L]) = p((0,L)) + \big[ w - u \big]^L_0 = \frac{\sigma l }{a_0} = \frac{\sigma \mu ([0,L])}{a_0}.
$$
On the other hand, by \eqref{eq:dot mu dot l} we infer that for $\LL^1$-a.e. time in $ [0,T]$, the following time derivatives exist and satisfy 
\begin{equation}\label{eq:dot p}
\dot p = \frac{ \dot \sigma \mu}{a_0} + \frac{\sigma \dot \mu}{a_0} \text{ in } \M([0,L]) \quad \text{hence} \quad \dot p ([0,L]) = \frac{ \dot \sigma l}{a_0} + \frac{\sigma \dot \mu ([0,L])}{a_0} =  \frac{ \dot{ \wideparen{ \sigma l }} }{a_0} = \dot{ \wideparen{ p([0,L])}} .
\end{equation}
Let $t_0 := \sup \, \{ t \in [0,T] \, : \, l(t) = 0 \}$, with the convention $t_0=0$ if $l(0)>0$ (see Figure \ref{fig:2}). As $l$ is non-decreasing on $[0,T]$, if $t_0 = T$ then $l \equiv 0$ at all time and \eqref{eq:2} entails that $\big[ w \big]^L_0 \equiv \sigma \frac{L}{a_1}$ at all time as well, so that \eqref{eq:CNS w} is always true since $\sigma \in \KK$. We can thus assume that $t_0 < T$. Suppose that $0 \leq a < b \leq T$ are such that
$$\left\lvert \big[ w(b) \big]^L_0 \right\rvert < \left\lvert \big[ w(a) \big]^L_0 \right\rvert .$$
We have to prove that 
\begin{equation}\label{eq:NC}
\left\lvert \big[ w(b) \big]^L_0 \right\rvert \leq \sqrt{2\kappa a_0} \frac{L}{a_1}.
\end{equation}

Using again the non-decreasing character of $l$ and \eqref{eq:2}, since 
$$\left\lvert \big[ w \big]^L_0 \right\rvert = \left\lvert \sigma \right\rvert \left( \frac{l}{a_0} + \frac{L}{a_1} \right),$$
we necessarily have $ \left\lvert \sigma(b) \right\rvert < \left\lvert \sigma (a) \right\rvert \leq \sqrt{ 2 \kappa a_0}$ hence $\left\lvert \sigma(b) \right\rvert < \sqrt{ 2 \kappa a_0}$ in particular. By continuity of $\sigma$ and because $0 < b \leq T$, there exists $\eta > 0$ such that for all $s \in [b - \eta, b] \subset [0,T]$, $\left\lvert \sigma(s) \right\rvert < \sqrt{ 2 \kappa a_0}$. Let us consider 
$$ s_0 := \inf \{ s < b \, : \,  \left\lvert \sigma (t) \right\rvert < \sqrt{2 \kappa a_0} \quad \text{for all } s < t \leq b \} $$
and
$$
I := (s_0,b] .
$$
Recalling Theorem \ref{thm:DAMAGE}, by continuity we know that $l$ is constant on the whole segment 
$$\bar{I} := [s_0,b].$$
Let us prove that $l \equiv 0$ on $\bar I$ (hence on $[0,b]$), which will prove \eqref{eq:NC} due to \eqref{eq:2} once more. By contradiction, assume that $l \equiv c >0$ on $\bar I$. Since $\sigma$ does not saturate the constraint during all the time interval $(s_0,b]$, the Flow Rule \eqref{eq:FR} implies that $\dot p(s) ([0,L]) = \frac{ \dot \sigma (s) l  }{a_0} = 0$ for $\LL^1$-a.e. $s \in (s_0,b]$. Therefore, $\sigma$ is constant during the whole time interval $\bar I$ and in particular
$$\left\lvert \sigma(s_0) \right\rvert < \sqrt{ 2 \kappa a_0}.$$
Next, let us show that $s_0 > 0$. Indeed, 
\begin{itemize}
\item either $l(0)=0$. Hence, one can check that $l(t_0)=0$ by definition of $t_0$ and continuity of $l$. This implies that $s_0 > t_0$ since $l (s_0) = c >0$. 
\item Or $l(0) > 0$ thus $t_0=0$. Then, \eqref{eq:Constitutive Equation Initial Time} entails that $\left\lvert \sigma(0) \right\rvert = \sqrt{ 2 \kappa a_0}$ so that $s_0 > 0$ since $\left\lvert \sigma(s_0) \right\rvert < \sqrt{2 \kappa a_0}$.
\end{itemize}
Thus, since $s_0 > 0$, $\left\lvert \sigma(s_0) \right\rvert < \sqrt{2 \kappa a_0}$ and $\sigma$ is continuous on $[0,T]$, there exists $\eta > 0$ small enough such that $\left\lvert \sigma(s) \right\rvert < \sqrt{ 2 \kappa a_0}$ for all $s \in (s_0 - \eta, b] \subset [0,T]$, which is impossible by definition of $I$. 

\medskip
{\it Sufficient condition.} Assume \eqref{eq:CNS w}. Let us show that the Energy Balance \eqref{eq:Energy Balance} is satisfied. As before, we consider the time $t_0 := \sup \, \{ t \in [0,T] \, : \, l(t) = 0 \}$, with the convention $t_0=0$ if $l(0)>0$. If $t_0 = T$ then $l \equiv 0$ at all time and \eqref{eq:link mu l} entails that $\mu \equiv 0 \equiv p$ in $\M([0,L])$ at all time as well. Hence, the Energy Balance is obviously satisfied as it simply states that 
$$\E(t) = \E(0) + \int_0^t \int_0^L \sigma \dot w' \, dx ds = \frac{\sigma(t)^2 L}{2 a_1},$$
which is true thanks to \eqref{eq:E(t)} and \eqref{eq:2}. Therefore, we can assume that $t_0 < T$.

\medskip
Let us first note that we always have
\begin{equation}\label{eq: sigma t0}
\left\lvert \sigma(t_0) \right\rvert = \sqrt{ 2 \kappa a_0}.
\end{equation}
Indeed,
\begin{itemize}
\item either $t_0 > 0$. In this case, by continuity of $l$ and definition of $t_0$, we must have $l(t_0) = 0$. By contradiction, assume that $\left\lvert \sigma(t_0) \right\rvert < \sqrt{ 2 \kappa a_0}$. Then, by continuity of $\sigma$ and $l$ again, there exists $\eta > 0$ small enough such that $\left\lvert \sigma(s) \right\rvert < \sqrt{ 2 \kappa a_0} $ for all $s \in ( t_0 - \eta, t_0 + \eta) \subset [0,T]$. Once more, Theorem \ref{thm:DAMAGE} entails that $l$ is constant on the whole segment $[ t_0 - \eta, t_0 + \eta]$. In particular, $0 = l(t_0) = l(t_0 + \eta)$ which is impossible by definition of $t_0$.
\item Or $t_0 = 0$. We must again consider two cases.
		\begin{itemize}
		\item Either $l(0) = 0$. By contradiction, assume that $\left\lvert \sigma(0) \right\rvert < \sqrt{ 2 \kappa a_0} $. By continuity of $l$ and $\sigma$, as before there exists $\eta > 0$ small enough such that $\left\lvert \sigma(s) \right\rvert < \sqrt{ 2 \kappa a_0} $ for all $s \in [0, \eta) \subset [0,T]$, entailing that $0 = l(0) = l(\eta)$ which is again impossible by definition of $t_0$.
		\item Or $l(0) > 0$ and \eqref{eq:Constitutive Equation Initial Time} ensures that $\left\lvert \sigma(0) \right\rvert = \sqrt{ 2 \kappa a_0} $.
		\end{itemize}
\end{itemize} 

\medskip
Next, let us show that, as illustrated in Figure \ref{fig:2},
\begin{equation}\label{eq:necessary}
\left\lvert \sigma(t) \right\rvert = \sqrt{2 \kappa a_0} \quad \text{for all subsequent time } t \in [t_0,T].
\end{equation}
\begin{figure}[hbtp]
\begin{tikzpicture}[line cap=round,line join=round,>=triangle 45,x=0.8cm,y=0.8cm]
\draw [line width=1.pt] (0.,7.)-- (0.,-1.);
\draw [line width=1.pt] (-1,0.)-- (13.,0.);
\draw [line width=1.pt] (0.,7.)-- (-0.1,6.8);
\draw [line width=1.pt] (0.,7.)-- (0.1,6.8);
\draw [line width=1.pt] (13.,0.)-- (12.8,0.1);
\draw [line width=1.pt] (13.,0.)-- (12.8,-0.1);
\draw [line width=1.pt,dash pattern=on 5pt off 5pt] (-0.5,5.)-- (11.,5.);
\draw [line width=1.pt,dash pattern=on 5pt off 5pt] (11.,5.)-- (11.,-0.2);
\draw [line width=1.pt,dash pattern=on 5pt off 5pt] (5.,5.)-- (5.,-0.2);
\draw [line width=3.pt] (5.,5.)-- (11.,5.);
\draw (4.7,-0.1) node[anchor=north west] {$\mathlarger{\mathlarger{t_0}}$};
\draw (10.7,-0.1) node[anchor=north west] {$\mathlarger{\mathlarger{T}}$};
\draw (0.,-0.1) node[anchor=north west] {$\mathlarger{\mathlarger{0}}$};
\draw (-2.3,5.5) node[anchor=north west] {$\mathlarger{\mathlarger{\sqrt{ 2 \kappa a_0}}}$};
\draw (0.2,7.5) node[anchor=north west] {$\mathlarger{\mathlarger{\left\lvert \sigma(t) \right\rvert}}$};
\draw (6.5,3.) node[anchor=north west] {$\mathlarger{\mathlarger{\left\lvert \sigma \right\rvert \equiv \sqrt{2 \kappa a_0}}}$};
\draw (6.5,2.) node[anchor=north west] {$\mathlarger{\mathlarger{l > 0}}$};
\draw (1.5,4.) node[anchor=north west] {$\mathlarger{\mathlarger{l \equiv 0}}$};
\draw (1.5,3.) node[anchor=north west] {$\mathlarger{\mathlarger{\mu \equiv 0 \equiv p}}$};
\draw (1.5,2.) node[anchor=north west] {$\mathlarger{\mathlarger{\sigma  }}$ \large{ free}};
\draw (13.,0.35) node[anchor=north west] {\large{time}};
\end{tikzpicture}
\caption{}
\label{fig:2}
\end{figure}

\noindent
By contradiction, assume there exists $t_0 < t \leq T$ such that $\left\lvert \sigma(t) \right\rvert < \sqrt{2 \kappa a_0}$ and consider the maximal time interval where $\sigma$ does not saturate the constraint, by defining:
$$t_1 = \inf \{ s \in [0,t] \, : \, \left\lvert \sigma \right\rvert < \sqrt{2 \kappa a_0} \text{ for all time in } (s,t] \} $$
and
$$ t_2 = \sup \{ s \in [t,T] \, : \, \left\lvert \sigma \right\rvert < \sqrt{2 \kappa a_0} \text{ for all time in } [t,s) \}.$$
By continuity of $\sigma$ and \eqref{eq: sigma t0}, we have that $t_0 \leq t_1 < t \leq t_2 \leq T$ and $\left\lvert \sigma (s) \right\rvert < \sqrt{2 \kappa a_0}$ for all $s \in (t_1,t_2)$. Since $t_1 \geq t_0$, we infer that $l \equiv l(t_1)>0$ is a positive constant on the whole segment $[t_1,t_2]$ by Theorem \ref{thm:DAMAGE}. One can also check that $\left\lvert \sigma(t_1) \right\rvert = \sqrt{ 2 \kappa a_0}$. Indeed, either $t_1 = t_0$ and \eqref{eq: sigma t0} concludes, or $t_1 > t_0$. In particular $t_1 > 0$ and the continuity of $\sigma$ together with the definition of $t_1$ imply that  $\left\lvert \sigma(t_1) \right\rvert = \sqrt{ 2 \kappa a_0}$. Therefore, \eqref{eq:2} yields that for all $s \in (t_1,t_2)$
$${\displaystyle \left\lvert \big[ w(t_1) \big]^L_0 \right\rvert = \sqrt{2 \kappa a_0} \left( \frac{l}{a_0} + \frac{L}{a_1} \right) > \left\lvert \big[ w(s) \big]^L_0 \right\rvert = \left\lvert \sigma(s) \right\rvert \left( \frac{l}{a_0} + \frac{L}{a_1} \right) .}$$
By \eqref{eq:CNS w}, we get that ${\displaystyle \left\lvert \big[ w(s) \big]^L_0 \right\rvert  \leq \sqrt{ 2 \kappa a_0} \frac{L}{a_1} }$ for all $s \in (t_1,t_2)$. The continuity of $\big[ w \big]^L_0$ and \eqref{eq:2} imply that 
$${\displaystyle \left\lvert \big[ w(t_1) \big]^L_0 \right\rvert = \sqrt{2 \kappa a_0} \left( \frac{l}{a_0} + \frac{L}{a_1} \right) \leq \sqrt{ 2 \kappa a_0} \frac{L}{a_1} },$$
which is impossible by positivity of $l>0$. We have thus proven the validity of \eqref{eq:necessary}.

\medskip
Therefore, we are in the configuration of Figure \ref{fig:2}. Whether $l(0)>0$ or not, one can check that 
\begin{equation}\label{eq:almost done}
p(t) = \frac{\sigma(t_0)}{a_0} \mu(t) \quad \text{and} \quad \left( \sqrt{2 \kappa a_0} - \left\lvert \sigma(t) \right\rvert \right)^2 \frac{l(t)}{2a_0}  =0 \quad \text{for all } t \in [0,T].
\end{equation}
Indeed, either $l(t_0)>0$ hence $t_0=0$ and $\sigma = \sigma(0) = \pm \sqrt{ 2 \kappa a_0}$ is constant on the whole segment $[0,T]$. Or $l(t_0)=0$, so that $\mu(t) = 0 = p(t)$ for all $t \in [0,t_0]$ and $\sigma = \sigma(t_0) = \pm \sqrt{2 \kappa a_0}$ is constant on the whole segment $[t_0,T]$. In particular \eqref{eq:almost done} is satisfied. Consequently, the monotonicity of $\mu$ entails that
$$
\mathcal V (p;0,t) = \left\lvert \frac{\sigma(t_0)}{a_0} \right\rvert \left( \mu(t) \big( [0,L] \big) - \mu(0) \big( [0,L] \big) \right) = \left\lvert p(t) \right\rvert \big( [0,L] \big) - \left\lvert p(0) \right\rvert \big( [0,L] \big)
$$
for all $t \in [0,T]$. Using \eqref{eq:Initial Time}, \eqref{eq:E 0}, \eqref{eq:Constitutive Equation Initial Time} and \eqref{eq:E(t)}, we infer that 
\begin{multline*}
\E(t) = \frac{La_1}{2} e(0)^2 + \sqrt{2 \kappa a_0} \left\lvert p(0) \right\rvert \big([0,L] \big) + \int_0^t \int_0^L \sigma \dot w ' \, dx ds \\
= \frac{La_1}{2} e(t)^2 + \sqrt{2 \kappa a_0} \left\lvert p(t) \right\rvert \big( [0,L] \big) + \left( \sqrt{2 \kappa a_0} - \left\lvert \sigma(t) \right\rvert \right)^2 \frac{l(t)}{2a_0} \\
= \frac{La_1}{2} e(t)^2 + \sqrt{2 \kappa a_0} \left\lvert p(t) \right\rvert \big( [0,L] \big)
\end{multline*}
for all $t \in [0,T]$, which concludes the proof of the Energy Balance \eqref{eq:Energy Balance} since
\begin{multline*}
\frac{La_1}{2} e(0)^2  + \int_0^t \int_0^L \sigma \dot w ' \, dx ds = \frac{La_1}{2} e(t)^2 + \sqrt{2 \kappa a_0} \left\lvert p(t) \right\rvert \big( [0,L] \big) -  \sqrt{2 \kappa a_0} \left\lvert p(0) \right\rvert \big([0,L] \big) \\
 = \frac{La_1}{2} e(t)^2 + \sqrt{2 \kappa a_0} \mathcal{V}(p;0,t).
\end{multline*}
\end{proof}

\begin{rem} Condition \eqref{eq:CNS w} is equivalent to the non-decreasing character of $ \left\lvert \big[ w \big]^L_0 \right\rvert $ during the time interval $[t^*_0,T]$ where
$${\displaystyle t^*_0 = \inf \, \left\{ t \in [0,T] \, : \, \left\lvert \big[ w(t) \big]^L_0 \right\rvert > \sqrt{2\kappa a_0} L /a_1 \right\} },$$
with the convention $t^*_0 = T$ if $\left\lvert \big[ w \big]^L_0 \right\rvert$ remains smaller or equal to $ \sqrt{2\kappa a_0} L /a_1$ during the whole time interval $[0,T]$. Note that  
\begin{equation}\label{eq:t0 et t*0}
t^*_0 = t_0.
\end{equation}
Indeed, on the one hand, either $t_0 = 0 \leq t^*_0$, or $t_0 > 0$. In this case, we infer that 
$$\left\lvert \big[ w(t) \big]^L_0 \right\rvert = \left\lvert \sigma(t) \right\rvert \frac{L}{a_1} \leq \sqrt{2 \kappa a_0} \frac{L}{a_1}$$
for all previous time $0 \leq t < t_0$ according to \eqref{eq:2}, Proposition \ref{cor:subsequence} and the fact that $l(t) = 0$. Therefore, $t \leq t^*_0$ which leads to $t_0 \leq t^*_0$ when $t$ tends to $t_0$. On the other hand, assume by contradiction that $t^*_0 > t_0$. By definition of $t_0$ and $t^*_0$, we deduce that for all $t \in (t_0,t^*_0)$,
$$
\left\lvert \big[ w(t) \big]^L_0 \right\rvert = \left\lvert \sigma(t) \right\rvert \left( \frac{l(t)}{a_0} + \frac{L}{a_1} \right) \leq \sqrt{2 \kappa a_0} \frac{L}{a_1} \quad \text{and} \quad l(t) > 0 .
$$
In particular, we infer that $\left\lvert \sigma(t) \right\rvert < \sqrt{2 \kappa a_0}$ for all $t \in (t_0,t^*_0)$, entailing that $ l \equiv l(t^*_0) > 0 $ on $ [t_0,t^*_0]$ by \eqref{thm:DAMAGE} and continuity of $l$. Then, \eqref{eq: sigma t0} implies that 
$$
\left\lvert \big[ w(t_0) \big]^L_0 \right\rvert = \sqrt{2 \kappa a_0} \left( \frac{l(t^*_0)}{a_0} + \frac{L}{a_1} \right) > \sqrt{2 \kappa a_0} \frac{L}{a_1}
$$
which is impossible.
\end{rem}

\section{Concluding remarks} \label{section:ccl}
In spite of the conjecture motivated by the static analysis of \cite{BIR}, Theorem \ref{thm:CNS FR} determines the exact conditions on which the quasi-static evolution $(u,e,p,\sigma)$ is of perfect plasticity type or not. In particular, when the prescribed boundary datum $w \in AC([0,T];H^1(\R))$ is such that $\lvert [ w ]^L_0 \rvert$ is decreasing and remains larger than $\sqrt{2 \kappa a_0} \frac{L}{a_1}$, the Energy Balance \eqref{eq:Energy Balance} is never satisfied (see Figure \ref{fig:3}). 

\begin{figure}[hbtp]
\begin{tikzpicture}[line cap=round,line join=round,>=triangle 45,x=1.0cm,y=1.0cm]
\clip(-2.8,-0.4) rectangle (8.8,5);
\draw [line width=1.pt] (0.,5.)-- (0.,-0.5);
\draw [line width=1.pt] (-0.5,0.)-- (7.5,0.);
\draw [line width=1.pt] (0.,5.)-- (-0.1,4.8);
\draw [line width=1.pt] (0.,5.)-- (0.1,4.8);
\draw [line width=1.pt] (7.5,0.)-- (7.3,0.1);
\draw [line width=1.pt] (7.5,0.)-- (7.3,-0.1);
\draw [line width=1.pt,dash pattern=on 4pt off 4pt] (-0.3,1.)-- (6.9,1.);
\draw [line width=1.pt,dash pattern=on 4pt off 4pt] (6.9,1.)-- (6.9,-0.1);
\draw (6.6,0.) node[anchor=north west] {$\mathlarger{\mathlarger{T}}$};
\draw (0,0) node[anchor=north west] {$\mathlarger{\mathlarger{0}}$};
\draw (-1.9,2.2) node[anchor=north west] {$\mathlarger{\mathlarger{\sqrt{ 2 \kappa a_0} \frac{L}{a_1}}}$};
\draw (0.4,5.) node[anchor=north west] {$\mathlarger{\mathlarger{\left\lvert \big[ w(t) \big]^L_0 \right\rvert}}$};
\draw (7.6,0.3) node[anchor=north west] {\large{time}};
\draw[line width=1.2pt] (0.01682648858733636,4.093099359924541) -- (0.0715858523629325,3.9487044726589304);
\draw[line width=1.2pt] (0.0715858523629325,3.9487044726589304) -- (0.12634521613852862,3.788203430428377);
\draw[line width=1.2pt] (0.12634521613852862,3.788203430428377) -- (0.18110457991412474,3.6694640531741687);
\draw[line width=1.2pt] (0.18110457991412474,3.6694640531741687) -- (0.23586394368972086,3.583141428024387);
\draw[line width=1.2pt] (0.23586394368972086,3.583141428024387) -- (0.290623307465317,3.5186346375335122);
\draw[line width=1.2pt] (0.290623307465317,3.5186346375335122) -- (0.34538267124091315,3.468932478661655);
\draw[line width=1.2pt] (0.34538267124091315,3.468932478661655) -- (0.4001420350165093,3.4295885783212476);
\draw[line width=1.2pt] (0.4001420350165093,3.4295885783212476) -- (0.45490139879210545,3.3977253938763416);
\draw[line width=1.2pt] (0.45490139879210545,3.3977253938763416) -- (0.5096607625677015,3.3714213251963288);
\draw[line width=1.2pt] (0.5096607625677015,3.3714213251963288) -- (0.5644201263432976,3.3493523307055755);
\draw[line width=1.2pt] (0.5644201263432976,3.3493523307055755) -- (0.6191794901188937,3.3305797051602326);
\draw[line width=1.2pt] (0.6191794901188937,3.3305797051602326) -- (0.6739388538944898,3.3144208850268204);
\draw[line width=1.2pt] (0.6739388538944898,3.3144208850268204) -- (0.7286982176700859,3.300368378293086);
\draw[line width=1.2pt] (0.7286982176700859,3.300368378293086) -- (0.783457581445682,3.288037386264873);
\draw[line width=1.2pt] (0.783457581445682,3.288037386264873) -- (0.8382169452212781,3.277131039606696);
\draw[line width=1.2pt] (0.8382169452212781,3.277131039606696) -- (0.8929763089968742,3.2674167544538406);
\draw[line width=1.2pt] (0.8929763089968742,3.2674167544538406) -- (0.9477356727724703,3.2587097929020197);
\draw[line width=1.2pt] (0.9477356727724703,3.2587097929020197) -- (1.0024950365480665,3.250861602894231);
\draw[line width=1.2pt] (1.0024950365480665,3.250861602894231) -- (1.0572544003236626,3.243751397915199);
\draw[line width=1.2pt] (1.0572544003236626,3.243751397915199) -- (1.1120137640992587,3.2372799763967235);
\draw[line width=1.2pt] (1.1120137640992587,3.2372799763967235) -- (1.1667731278748548,3.2313651174022544);
\draw[line width=1.2pt] (1.1667731278748548,3.2313651174022544) -- (1.2215324916504509,3.225938103941692);
\draw[line width=1.2pt] (1.2215324916504509,3.225938103941692) -- (1.276291855426047,3.2209410651126746);
\draw[line width=1.2pt] (1.276291855426047,3.2209410651126746) -- (1.331051219201643,3.2163249210444342);
\draw[line width=1.2pt] (1.331051219201643,3.2163249210444342) -- (1.3858105829772391,3.212047777250789);
\draw[line width=1.2pt] (1.3858105829772391,3.212047777250789) -- (1.4405699467528352,3.2080736579596216);
\draw[line width=1.2pt] (1.4405699467528352,3.2080736579596216) -- (1.4953293105284313,3.204371497894673);
\draw[line width=1.2pt] (1.4953293105284313,3.204371497894673) -- (1.5500886743040274,3.200914333095444);
\draw[line width=1.2pt] (1.5500886743040274,3.200914333095444) -- (1.6048480380796235,3.197678646452276);
\draw[line width=1.2pt] (1.6048480380796235,3.197678646452276) -- (1.6596074018552196,3.1946438345514663);
\draw[line width=1.2pt] (1.6596074018552196,3.1946438345514663) -- (1.7143667656308157,3.191791770411855);
\draw[line width=1.2pt] (1.7143667656308157,3.191791770411855) -- (1.7691261294064118,3.18910644259767);
\draw[line width=1.2pt] (1.7691261294064118,3.18910644259767) -- (1.8238854931820079,3.1865736555985684);
\draw[line width=1.2pt] (1.8238854931820079,3.1865736555985684) -- (1.878644856957604,3.1841807796866495);
\draw[line width=1.2pt] (1.878644856957604,3.1841807796866495) -- (1.9334042207332,3.1819165409815877);
\draw[line width=1.2pt] (1.9334042207332,3.1819165409815877) -- (1.9881635845087962,3.179770844386147);
\draw[line width=1.2pt] (1.9881635845087962,3.179770844386147) -- (2.0429229482843922,3.1777346235446764);
\draw[line width=1.2pt] (2.0429229482843922,3.1777346235446764) -- (2.0976823120599883,3.175799713135665);
\draw[line width=1.2pt] (2.0976823120599883,3.175799713135665) -- (2.1524416758355844,3.1739587397161704);
\draw[line width=1.2pt] (2.1524416758355844,3.1739587397161704) -- (2.2072010396111805,3.1722050280501635);
\draw[line width=1.2pt] (2.2072010396111805,3.1722050280501635) -- (2.2619604033867766,3.170532520418933);
\draw[line width=1.2pt] (2.2619604033867766,3.170532520418933) -- (2.3167197671623727,3.168935706863011);
\draw[line width=1.2pt] (2.3167197671623727,3.168935706863011) -- (2.371479130937969,3.1674087248628964);
\draw[line width=1.2pt] (2.371479130937969,3.1674087248628964) -- (2.426238494713565,3.1655866531282655);
\draw[line width=1.2pt] (2.426238494713565,3.1655866531282655) -- (2.480997858489161,3.1605651716906564);
\draw[line width=1.2pt] (2.480997858489161,3.1605651716906564) -- (2.535757222264757,3.1488273716537907);
\draw[line width=1.2pt] (2.535757222264757,3.1488273716537907) -- (2.590516586040353,3.1299136111634906);
\draw[line width=1.2pt] (2.590516586040353,3.1299136111634906) -- (2.6452759498159493,3.10443570711131);
\draw[line width=1.2pt] (2.6452759498159493,3.10443570711131) -- (2.7000353135915454,3.0706579391296547);
\draw[line width=1.2pt] (2.7000353135915454,3.0706579391296547) -- (2.7547946773671415,3.0276799270485526);
\draw[line width=1.2pt] (2.7547946773671415,3.0276799270485526) -- (2.8095540411427375,2.977220210997849);
\draw[line width=1.2pt] (2.8095540411427375,2.977220210997849) -- (2.8643134049183336,2.921972107522837);
\draw[line width=1.2pt] (2.8643134049183336,2.921972107522837) -- (2.9190727686939297,2.864386277402378);
\draw[line width=1.2pt] (2.9190727686939297,2.864386277402378) -- (2.973832132469526,2.806306498983214);
\draw[line width=1.2pt] (2.973832132469526,2.806306498983214) -- (3.028591496245122,2.7489924285119383);
\draw[line width=1.2pt] (3.028591496245122,2.7489924285119383) -- (3.083350860020718,2.6932484751157815);
\draw[line width=1.2pt] (3.083350860020718,2.6932484751157815) -- (3.138110223796314,2.6395537172774293);
\draw[line width=1.2pt] (3.138110223796314,2.6395537172774293) -- (3.19286958757191,2.588164892591469);
\draw[line width=1.2pt] (3.19286958757191,2.588164892591469) -- (3.2476289513475063,2.539191038592742);
\draw[line width=1.2pt] (3.2476289513475063,2.539191038592742) -- (3.3023883151231024,2.492645392805161);
\draw[line width=1.2pt] (3.3023883151231024,2.492645392805161) -- (3.3571476788986985,2.4484807158214177);
\draw[line width=1.2pt] (3.3571476788986985,2.4484807158214177) -- (3.4119070426742946,2.406613031931171);
\draw[line width=1.2pt] (3.4119070426742946,2.406613031931171) -- (3.4666664064498907,2.3669374389891056);
\draw[line width=1.2pt] (3.4666664064498907,2.3669374389891056) -- (3.5214257702254868,2.3293385407092586);
\draw[line width=1.2pt] (3.5214257702254868,2.3293385407092586) -- (3.576185134001083,2.2936972497836026);
\draw[line width=1.2pt] (3.576185134001083,2.2936972497836026) -- (3.630944497776679,2.2598951480026006);
\draw[line width=1.2pt] (3.630944497776679,2.2598951480026006) -- (3.685703861552275,2.227817205361995);
\draw[line width=1.2pt] (3.685703861552275,2.227817205361995) -- (3.740463225327871,2.197353400151276);
\draw[line width=1.2pt] (3.740463225327871,2.197353400151276) -- (3.795222589103467,2.1683996066611);
\draw[line width=1.2pt] (3.795222589103467,2.1683996066611) -- (3.8499819528790633,2.1408579988468044);
\draw[line width=1.2pt] (3.8499819528790633,2.1408579988468044) -- (3.9047413166546594,2.1146371383133165);
\draw[line width=1.2pt] (3.9047413166546594,2.1146371383133165) -- (3.9595006804302555,2.089651860766807);
\draw[line width=1.2pt] (3.9595006804302555,2.089651860766807) -- (4.014260044205852,2.0658230382017773);
\draw[line width=1.2pt] (4.014260044205852,2.0658230382017773) -- (4.069019407981448,2.043077268937367);
\draw[line width=1.2pt] (4.069019407981448,2.043077268937367) -- (4.123778771757044,2.021346530416533);
\draw[line width=1.2pt] (4.123778771757044,2.021346530416533) -- (4.17853813553264,2.000567817902792);
\draw[line width=1.2pt] (4.17853813553264,2.000567817902792) -- (4.233297499308236,1.9806827841397654);
\draw[line width=1.2pt] (4.233297499308236,1.9806827841397654) -- (4.288056863083832,1.9616373895169623);
\draw[line width=1.2pt] (4.288056863083832,1.9616373895169623) -- (4.342816226859428,1.943381568518567);
\draw[line width=1.2pt] (4.342816226859428,1.943381568518567) -- (4.397575590635024,1.9258689156773072);
\draw[line width=1.2pt] (4.397575590635024,1.9258689156773072) -- (4.45233495441062,1.909056392539);
\draw[line width=1.2pt] (4.45233495441062,1.909056392539) -- (4.507094318186216,1.892904056006738);
\draw[line width=1.2pt] (4.507094318186216,1.892904056006738) -- (4.5618536819618125,1.877374807697739);
\draw[line width=1.2pt] (4.5618536819618125,1.877374807697739) -- (4.616613045737409,1.8624341634860988);
\draw[line width=1.2pt] (4.616613045737409,1.8624341634860988) -- (4.671372409513005,1.848050042134218);
\draw[line width=1.2pt] (4.671372409513005,1.848050042134218) -- (4.726131773288601,1.8341925717739445);
\draw[line width=1.2pt] (4.726131773288601,1.8341925717739445) -- (4.780891137064197,1.8208339129432467);
\draw[line width=1.2pt] (4.780891137064197,1.8208339129432467) -- (4.835650500839793,1.807948096886006);
\draw[line width=1.2pt] (4.835650500839793,1.807948096886006) -- (4.890409864615389,1.7955108778608078);
\draw[line width=1.2pt] (4.890409864615389,1.7955108778608078) -- (4.945169228390985,1.783499598265244);
\draw[line width=1.2pt] (4.945169228390985,1.783499598265244) -- (4.999928592166581,1.7718930654554752);
\draw[line width=1.2pt] (4.999928592166581,1.7718930654554752) -- (5.054687955942177,1.760671439219979);
\draw[line width=1.2pt] (5.054687955942177,1.760671439219979) -- (5.109447319717773,1.749816128947118);
\draw[line width=1.2pt] (5.109447319717773,1.749816128947118) -- (5.1642066834933695,1.7393096996055246);
\draw[line width=1.2pt] (5.1642066834933695,1.7393096996055246) -- (5.218966047268966,1.7291357857324114);
\draw[line width=1.2pt] (5.218966047268966,1.7291357857324114) -- (5.273725411044562,1.719279012696786);
\draw[line width=1.2pt] (5.273725411044562,1.719279012696786) -- (5.328484774820158,1.7097249245715398);
\draw[line width=1.2pt] (5.328484774820158,1.7097249245715398) -- (5.383244138595754,1.7004599180102935);
\draw[line width=1.2pt] (5.383244138595754,1.7004599180102935) -- (5.43800350237135,1.6914711815817043);
\draw[line width=1.2pt] (5.43800350237135,1.6914711815817043) -- (5.492762866146946,1.6827466400658544);
\draw[line width=1.2pt] (5.492762866146946,1.6827466400658544) -- (5.547522229922542,1.6742749032645547);
\draw[line width=1.2pt] (5.547522229922542,1.6742749032645547) -- (5.602281593698138,1.6660452189202308);
\draw[line width=1.2pt] (5.602281593698138,1.6660452189202308) -- (5.657040957473734,1.6580474293768273);
\draw[line width=1.2pt] (5.657040957473734,1.6580474293768273) -- (5.7118003212493305,1.6502719316512033);
\draw[line width=1.2pt] (5.7118003212493305,1.6502719316512033) -- (5.7665596850249266,1.6427096406150952);
\draw[line width=1.2pt] (5.7665596850249266,1.6427096406150952) -- (5.821319048800523,1.6353519550162474);
\draw[line width=1.2pt] (5.821319048800523,1.6353519550162474) -- (5.876078412576119,1.6281907260929813);
\draw[line width=1.2pt] (5.876078412576119,1.6281907260929813) -- (5.930837776351715,1.6212182285596493);
\draw[line width=1.2pt] (5.930837776351715,1.6212182285596493) -- (5.985597140127311,1.6144271337612413);
\draw[line width=1.2pt] (5.985597140127311,1.6144271337612413) -- (6.040356503902907,1.6078104848142236);
\draw[line width=1.2pt] (6.040356503902907,1.6078104848142236) -- (6.095115867678503,1.6013616735675864);
\draw[line width=1.2pt] (6.095115867678503,1.6013616735675864) -- (6.149875231454099,1.5950744192333621);
\draw[line width=1.2pt] (6.149875231454099,1.5950744192333621) -- (6.204634595229695,1.588942748549596);
\draw[line width=1.2pt] (6.204634595229695,1.588942748549596) -- (6.259393959005291,1.582960977351184);
\draw[line width=1.2pt] (6.259393959005291,1.582960977351184) -- (6.3141533227808875,1.5771236934351727);
\draw[line width=1.2pt] (6.3141533227808875,1.5771236934351727) -- (6.368912686556484,1.571425740617224);
\draw[line width=1.2pt] (6.368912686556484,1.571425740617224) -- (6.42367205033208,1.5658622038850964);
\draw[line width=1.2pt] (6.42367205033208,1.5658622038850964) -- (6.478431414107676,1.5604283955632323);
\draw[line width=1.2pt] (6.478431414107676,1.5604283955632323) -- (6.533190777883272,1.5551198424100354);
\draw[line width=1.2pt] (6.533190777883272,1.5551198424100354) -- (6.587950141658868,1.5499322735761645);
\draw[line width=1.2pt] (6.587950141658868,1.5499322735761645) -- (6.642709505434464,1.5448616093583123);
\draw[line width=1.2pt] (6.642709505434464,1.5448616093583123) -- (6.69746886921006,1.5399039506884993);
\draw[line width=1.2pt] (6.69746886921006,1.5399039506884993) -- (6.752228232985656,1.535055569303939);
\draw[line width=1.2pt] (6.752228232985656,1.535055569303939) -- (6.806987596761252,1.530312898547121);
\draw[line width=1.2pt] (6.806987596761252,1.530312898547121) -- (6.861746960536848,1.5256725247499197);
\draw[line width=1.2pt] (6.861746960536848,1.5256725247499197) -- (6.9165063243124445,1.5211311791593154);
\draw (1.2,0.63) node[anchor=north west] {\large{Flow Rule never satisfied}};
\end{tikzpicture}
\caption{}
\label{fig:3}
\end{figure}
This suggests to interpret $(u,c)$ rather as a quasi-static evolution of damage as stated in Theorem \ref{thm:DAMAGE}, even when the prescribed boundary datum satisfies \eqref{eq:CNS w}. In this case, the very specific nature of the plastic evolution illustrated in Figure \ref{fig:2} seems to confirm the interpretation of the evolution as one of damage. Indeed, the only configuration of perfect plasticity we obtain is very restrictive as the evolution remains purely elastic until a threshold time $t_0$, after which the stress $\sigma$ always saturates the constraint and the damage keeps on increasing, so that the elastic strain remains constant until the end of the process which is rather specific to damage than plasticity. Besides, when choosing $T > 2 \sqrt{2\kappa a_0}/a_1$ and applying a loading-unloading Dirichlet condition
$$
w : (t,x) \in [0,T] \times [0,L] \mapsto x \left( t \mathds{1}_{\left[0,\frac{T}{2} \right]}(t) + (T-t) \mathds{1}_{\left( \frac{T}{2}, T \right) } (t) \right),
$$
the response of the limit model is indeed typical of damage, as illustrated in Figure \ref{fig:loading unloading}.
\begin{figure}[hbtp]
\begin{tikzpicture}[x=0.34cm,y=0.34cm]
\clip(-14,-0.5) rectangle (52,33);
\draw [line width=1.pt] (0.,30.)-- (0.,18.);
\draw [line width=1.pt] (0.,12.)-- (0.,0.);
\draw [line width=1.pt] (-2.,20.)-- (22.,20.);
\draw [line width=1.pt] (-2.,2.)-- (22.,2.);
\draw [line width=0.5pt,dotted] (-1.,26.)-- (22.,26.);
\draw [line width=0.5pt,dotted] (-1.,10.)-- (22.,10.);
\draw [line width=1.pt] (10.011127673837626,29.01895684052162)-- (0.,20.);
\draw [line width=1.pt] (10.011127673837626,29.01895684052162)-- (20.,20.);
\draw [line width=0.5pt,dotted] (6.66005693398481,26.)-- (6.526069415276971,0.8939270110483071);
\draw [line width=0.5pt,dotted] (10.011127673837626,29.01895684052162)-- (9.969841269281858,0.8804251052739491);
\draw [line width=1.pt] (0.,30.)-- (-0.3,29.4);
\draw [line width=1.pt] (0.,30.)-- (0.3,29.4);
\draw [line width=1.pt] (22.,20.)-- (21.4,20.3);
\draw [line width=1.pt] (22.,20.)-- (21.4,19.7);
\draw [line width=1.pt] (0.,12.)-- (-0.3,11.4);
\draw [line width=1.pt] (0.,12.)-- (0.3,11.4);
\draw [line width=1.pt] (22.,2.)-- (21.4,2.3);
\draw [line width=1.pt] (22.,2.)-- (21.4,1.7);
\draw [line width=1.5pt,color=qqqqff] (0.,20.)-- (6.6280357926218905,20.);
\draw [line width=1.5pt,color=ffqqqq] (9.997894563301076,20.)-- (20.,20.);
\draw [line width=1.5pt,color=qqqqff] (0.,2.)-- (6.531972368533127,2.);
\draw [line width=1.5pt,color=ffqqqq] (9.97148397127033,2.)-- (20.,2.);
\draw [line width=1.pt] (0.,2.)-- (6.574667223683688,10.);
\draw [line width=1.pt] (6.574667223683688,10.)-- (9.983222012172881,10.);
\draw [line width=1.pt] (9.983222012172881,10.)-- (20.,2.);
\draw (-1.85,33) node[anchor=north west] {$\big[ w(t) \big]^L_0$};
\draw (-1.3,14.332010719735797) node[anchor=north west] {$\sigma(t)$};
\draw (8.8,31.3) node[anchor=north west] {$ \frac{TL}{2}$};
\draw (22.,21) node[anchor=north west] {$t$};
\draw (19.,20.) node[anchor=north west] {$T$};
\draw (6.5,20.) node[anchor=north west] {$t_0$};
\draw (9.8,20.) node[anchor=north west] {$\frac{T}{2}$};
\draw (0.,20.) node[anchor=north west] {$0$};
\draw (0.,2.) node[anchor=north west] {$0$};
\draw (6.5,2.) node[anchor=north west] {$t_0$};
\draw (9.8,2.) node[anchor=north west] {$\frac{T}{2}$};
\draw (19.,2) node[anchor=north west] {$T$};
\draw (22.,3.) node[anchor=north west] {$t$};
\draw [color=qqqqff](2.0051421534379403,20) node[anchor=north west] {$\mu \equiv 0$};
\draw [color=qqqqff](2.0051421534379403,2) node[anchor=north west] {$\mu \equiv 0$};
\draw [color=ffqqqq](12.795368775947116,20.) node[anchor=north west] {$\mu \equiv \mu\left( {\scriptstyle \frac{T}{2}} \right)$};
\draw [color=ffqqqq](12.850420952592572,2.) node[anchor=north west] {$\mu \equiv \mu\left( {\scriptstyle \frac{T}{2}} \right)$};
\draw (-6.748153933189402,27.1) node[anchor=north west] {$\sqrt{2 \kappa a_0} \frac{L}{a_1}$};
\draw (-5.867319106862122,11.) node[anchor=north west] {$\sqrt{2 \kappa a_0}$};
\draw (2.9,29.) node[anchor=north west] {Loading};
\draw (11.8,29.) node[anchor=north west] {Unloading};
\draw [line width=0.5pt] (2.019075312479522,4.456793925881183)-- (2.,6.);
\draw [line width=0.5pt] (2.,6.)-- (3.28451198190392,5.996566664329098);
\draw [line width=0.5pt] (14.996732936611028,5.995909318920065)-- (16.367471155191016,6.006868599090258);
\draw [line width=0.5pt] (16.367471155191016,6.006868599090258)-- (16.377516376656267,4.893132803978246);
\draw (0.5,7.9) node[anchor=north west] {$a_1$};
\draw (16.,8.) node[anchor=north west] {$- \frac{2 \sqrt{2 \kappa a_0}}{T}$};
\end{tikzpicture}
\caption{}
\label{fig:loading unloading}
\end{figure}

\noindent
Using \eqref{eq:t0 et t*0}, one can check that $t_0 = \sqrt{2 \kappa a_0}/a_1 > 0$, hence 
$$\mu \equiv 0 \quad \text{and} \quad \sigma \equiv \frac{a_1}{L} \big[ w \big]^L_0 \geq 0 \quad \text{on } [0,t_0].$$
By \eqref{eq: sigma t0}, \eqref{thm:DAMAGE} and the increasing character of $\big[ w \big]^L_0$ during the time interval $\left[t_0, \frac{T}{2} \right]$, we infer that
$$
\sigma \equiv \sqrt{2 \kappa a_0} \quad \text{on } \left[ t_0, \frac{T}{2} \right],
$$ 
which corresponds to the hardening phase. Then, since $l$ is non-decreasing and $\big[ w \big]^L_0$ is decreasing during the unloading phase, \eqref{eq:E(t)} entails that $\sigma$ decreases during the time interval $\left[ \frac{T}{2},T \right]$. In particular, $0 \leq \sigma(t) < \sqrt{2 \kappa a_0}$ for all $t \in \left( \frac{T}{2},T \right]$, so that $\mu \equiv \mu \left( \frac{T}{2} \right)$ is constant during the unloading phase and
$$
\sigma(t) \equiv \frac{ \big[w(t) \big]}{ \frac{l \left( \frac{T}{2} \right)}{a_0} + \frac{L}{a_1}} = 2 \sqrt{2 \kappa a_0} \frac{T-t}{T} \quad \text{for all } t \in \left[ \frac{T}{2}, T \right].
$$
In particular, at the end of the loading-unloading process, the medium goes back to its reference configuration, whereas in perfect plasticity one expects a residual plastic strain (see Figure \ref{fig:damage VS plasticity}). On the other hand, interpreting the evolution as one of damage also fails to be completely satisfactory, as $\left\lvert \sigma \right\rvert$ never exceeds the damage yield threshold $\sqrt{2 \kappa a_0}$, including during the hardening phase, which is specific to perfect plasticity and therefore consists in a painful lack of generality for the description of a damage evolution as well (see Figure \ref{fig:damage VS plasticity}). One could wonder if the effective model we obtained lies somehow in between damage and plasticity. Without being able to answer completely this question, let us remark that the effective evolution obtained here does not fit in the large class of elastoplasticity-damage models introduced in \cite{AMV,Cr1,Cr2} where we expect the Energy Balance to hold and the plastic yield surface to shrink as damage increases, whereas here \eqref{eq:Energy Balance} is not always satisfied and $\partial \KK =  \{ \pm \sqrt{2 \kappa a_0 } \}$ is fixed.

\begin{figure}[hbtp]
\begin{tikzpicture}[line cap=round,line join=round,>=triangle 45,x=0.315cm,y=0.315cm]
\clip(-4,-6) rectangle (48,12.);
\draw [line width=1.pt] (0.,-2.)-- (0.,10.);
\draw [line width=1.pt] (-0.6,0.)-- (9.5,0.);
\draw [line width=1.pt] (14.5,-2.)-- (14.5,10.); 
\draw [line width=1.pt] (14,0.)-- (24.5,0.); 
\draw [line width=1.pt] (-0.2,9.6)-- (0.,10.);
\draw [line width=1.pt] (0.2,9.6)-- (0.,10.);
\draw [line width=1.pt] (9.1,0.2)-- (9.5,0.);
\draw [line width=1.pt] (9.1,-0.2)-- (9.5,0.);
\draw [line width=1.pt] (14.3,9.6)-- (14.5,10.);
\draw [line width=1.pt] (14.7,9.6)-- (14.5,10.);
\draw [line width=1.pt] (24.1,0.2)-- (24.5,0.); 
\draw [line width=1.pt] (24.1,-0.2)-- (24.5,0.);
\draw [line width=1.pt] (0.,0.)-- (4.,8.);
\draw [line width=1.pt] (4.,8.)-- (8.,8.);
\draw [line width=1.pt] (8.,8.)-- (0.,0.);
\draw [line width=1.pt] (14.5,0.)-- (18.5,8.); 
\draw [line width=1.pt] (18.5,8.)-- (22.5,8.); 
\draw [line width=1.pt] (22.5,8.)-- (18.5,0.); 
\draw [line width=0.5pt] (2.3,5.2)-- (2.8390652861162837,5.6781305722325675); 
\draw [line width=0.5pt] (2.8390652861162837,5.6781305722325675)-- (2.846600633470912,5); 
\draw [line width=0.5pt] (5.6,8.3)-- (6.,8.); 
\draw [line width=0.5pt] (5.6,7.7)-- (6.,8.); 
\draw [line width=0.5pt] (5.8,6.2)-- (5.576126642533382,5.576126642533382);
\draw [line width=0.5pt] (6.15,5.75)-- (5.576126642533382,5.576126642533382);
\draw [line width=0.5pt] (16.7,5.1)-- (17.316380822082582,5.632761644165163); 
\draw [line width=0.5pt] (17.316380822082582,5.632761644165163)-- (17.3,4.85);
\draw [line width=0.5pt] (20.5,8.3)-- (21.021896241741526,8.); 
\draw [line width=0.5pt] (20.5,7.7)-- (21.021896241741526,8.);
\draw [line width=0.5pt] (20.4,4.5)-- (20.326620713215465,3.6532414264309327); 
\draw [line width=0.5pt] (21.,4.1)-- (20.326620713215465,3.6532414264309327); 
\draw (-0.9,12) node[anchor=north west] {$\sigma(t)$};
\draw (13.5,12) node[anchor=north west] {$\sigma(t)$};
\draw [line width=0.5pt,dotted] (-0.5,8.05038435200938)-- (40.,8.);
\draw (-4.5,9.075461464607216) node[anchor=north west] {$\sqrt{2\kappa a_0}$};
\draw (2.,-0.2820932650105007) node[anchor=north west] {$\frac{\sqrt{2 \kappa a_0} L}{a_1}$};
\draw (7.,-0.4770423218775365) node[anchor=north west] {$\frac{TL}{2}$};
\draw (16.7,-0.2820932650105007) node[anchor=north west] {$\frac{\sqrt{2 \kappa a_0} L}{a_1}$};
\draw (21.5,-0.4770423218775365) node[anchor=north west] {$\frac{TL}{2}$};
\draw (9.2,1.3) node[anchor=north west] {$\big[ w(t) \big]^L_0$};
\draw (24.3,1.3) node[anchor=north west] {$\big[ w(t) \big]^L_0$};
\draw (0.5,-3.) node[anchor=north west] {Effective model};
\draw (16.,-3) node[anchor=north west] {Plasticity};
\draw (32,-3) node[anchor=north west] {Damage};
\draw [line width=0.5pt] (1.,2.)-- (1.,3.); 
\draw [line width=0.5pt] (1.,3.)-- (1.5,3);
\draw [line width=0.5pt] (3.,3.)-- (3.8,3.);
\draw [line width=0.5pt] (3.8,3.)-- (3.8,3.8);
\draw [line width=0.5pt] (15.5,2.)-- (15.5,3.);
\draw [line width=0.5pt] (15.5,3.)-- (16,3);
\draw (0.1,5.) node[anchor=north west] {$\frac{a_1}{L}$};
\draw (3.9,4.) node[anchor=north west] {$\frac{2 \sqrt{2\kappa a_0}}{LT}$};
\draw (14.5,5.) node[anchor=north west] {$\frac{a_1}{L}$};
\draw [line width=1.pt] (30.,-2.)-- (30.,10.);
\draw [line width=1.pt] (29.5,0.)-- (40.,0.);
\draw [line width=1.pt] (29.8,9.6)-- (30.,10.);
\draw [line width=1.pt] (30.2,9.6)-- (30.,10.);
\draw [line width=1.pt] (39.6,0.2)-- (40.,0.);
\draw [line width=1.pt] (39.6,-0.2)-- (40.,0.);
\draw [line width=1.pt] (30.,0.)-- (34.,8.);
\draw [line width=1.pt] (34.,8.)-- (37.98512129669676,9.24390655243666);
\draw [line width=1.pt] (37.98512129669676,9.24390655243666)-- (30.,0.);
\draw [line width=0.5pt] (32.2,5.2)-- (32.88846791084805,5.776935821696096);
\draw [line width=0.5pt] (32.9,5)-- (32.88846791084805,5.776935821696096);
\draw [line width=0.5pt] (35.5,8.75)-- (36.122466843542135,8.662501895789225);
\draw [line width=0.5pt] (36.122466843542135,8.662501895789225)-- (35.7,8.2);
\draw [line width=0.5pt] (35.4,6.7)-- (35.19876466324249,6.018304913540057);
\draw [line width=0.5pt] (35.19876466324249,6.018304913540057)-- (35.9,6.3);
\draw (33,11) node[anchor=north west] {hardening};
\draw (29.098388882778064,12) node[anchor=north west] {$\sigma(t)$};
\draw (39.8,1.3) node[anchor=north west] {$\big[ w(t) \big]^L_0$};
\draw (32.217573792650626,-0.2820932650105007) node[anchor=north west] {$\frac{\sqrt{2 \kappa a_0} L}{a_1}$};
\draw (37.20826964844672,-0.4770423218775365) node[anchor=north west] {$\frac{TL}{2}$};
\draw [line width=0.5pt,dotted] (4.,8.)-- (4,-0.5);
\draw [line width=0.5pt,dotted] (8.,8.)-- (8,-0.5);
\draw [line width=0.5pt,dotted] (18.5,8.)-- (18.5,-0.5);
\draw [line width=0.5pt,dotted] (22.5,8.)-- (22.5,-0.5);
\draw [line width=0.5pt,dotted] (34.,8.)-- (34,-0.5);
\draw [line width=0.5pt,dotted] (38,9.243906552436654)-- (38,-0.5);
\end{tikzpicture}
\caption{}
\label{fig:damage VS plasticity}
\end{figure}

\medskip
Looking at the constructive proof of \cite[Theorem 2]{FG}, one could argue that passing to the limit $\e \searrow 0$ in the time-continuous quasi-static evolutions might not have been the right approach, as some incremental information (see the minimality formulae and track of the history of damage \cite[Formulae (15), (16), (21)]{FG}) available at the stage of time discretizations is lost once the time step has been sent to $0$ (see \cite{BCGS} for a related case study of non-commutability). Indeed, let us fix a time subdivision $\T_N = \{ 0 = t_0^N, ..., t_N^N = T \}$ of $[0,T]$ with
$$\delta_N := \underset{ i \in \llbracket 0,N-1 \rrbracket }{\sup} \, \left\lvert t_{i+1}^N - t_i^N \right\rvert.$$
Following minor adaptations of the present work and passing first to the limit $\e \searrow 0$, we infer the existence of a piecewise constant in time evolution
$$
(u_N,e_N,p_N,\sigma_N,\mu_N) : [0,T] \to BV((0,L)) \times L^2((0,L)) \times \M([0,L]) \times \R \times \M([0,L])
$$
with uniformly bounded variation in $N$ such that 
$$\sigma_N \in L^\infty([0,T];\KK) \quad \text{is homogeneous in space,} $$
$$
p_N = \frac{\sigma_N}{a_0} \mu_N, \quad e_N = \frac{\sigma_N}{a_1}, \quad \big[ w \big]^L_0 = \sigma_N \left( \frac{\mu_N \big( [0,L] \big)}{a_0} + \frac{L}{a_1} \right),$$
and for all $i \in \llbracket 0,N \rrbracket$
\begin{equation}\label{eq:incremental properties}
\begin{cases}
\ds \big(u_N^i, e_N^i, p_N^i \big) \in \A(w(t_i^N)), \\
\ds \E_N^i := \frac{\sigma_N^i \big[ w(t_i^N) \big]^L_0}{2} + \kappa \mu_N^i \big( [0,L] \big) =  \E(0) + \int_0^{t_i^N} \sigma_N \big[ \dot w \big]^L_0 \, ds + o_{\delta_N \searrow 0}(1), \\
\ds \mu_N^i \big( [0,L] \big) = \frac{a_0}{2 \kappa} \left(  \frac{\E_N^i}{a_0} - \frac{\kappa L}{a_1} + \sqrt{\Delta_N^i} \right),
\end{cases}
\end{equation}
where 
$$\Delta_N^i = \left( \frac{\E_N^i}{a_0} + \frac{\kappa L}{a_1} \right)^2 - \frac{2 \kappa}{a_0} \left\lvert \big[ w(t^N_i) \big]^L_0 \right\rvert^2 \geq 0$$
and $(v,\eta,q) \in \A(w(t_i^N))$ means that $v \in BV((0,L))$, $\eta \in L^2((0,L))$, $q \in \M \big( [0,L] \big)$, $Dv = \eta + q \res (0,L)$ and $q \res \{0,L\} = (w(t_i^N) - v) \big( \delta_L - \delta_0 \big)$. Moreover, as we pass to the limit $\e \searrow 0$ in the incremental minimality \cite[Formulae (15) and (16)]{FG}
$$
(u^i_{N,\e}, 1- \Theta^i_{N,\e}, a^i_{N,\e}) \in \underset{ \begin{gathered}
													  {\scriptstyle u \in w(t_i^N) + H^1_0((0,L))} \\
													  {\scriptstyle \theta \in L^\infty((0,L);[0,1])} \\
													  {\scriptstyle a \in \G_\theta(\e a_0,a^{i-1}_{N,\e})}
													\end{gathered} }{\rm argmin}  \left\{ \int_0^L \left( \frac12 a \left\lvert u' \right\rvert^2 + \frac{\kappa}{\e} \theta \Theta^{i-1}_{N,\e} \right) \, dx \right\}
$$
of the discrete evolution of \cite{FG}, one could hope that $(u_N,e_N,p_N)$ satisfies a stronger incremental minimality as in \cite[Formula (4.12)]{DMDSM}:
\begin{equation}\label{eq:better minimality}
(u^i_N, e^i_N,p^i_N) \in \underset{ (u,e,p) \in \A(w(t_i^N)) }{\rm argmin} \left\{ \int_0^L \frac12 a_1 e^2 \, dx + \sqrt{2 \kappa a_0} \left\lvert p - p^{i-1}_N \right\rvert \big( [0,L] \big) \right\}.
\end{equation}
Assume that \eqref{eq:better minimality} holds. Then, following exactly the proof of \cite[Theorem 4.5]{DMDSM} ensures the existence of a subsequence (independent of $t$, not relabeled) and a quasi-static evolution of perfect plasticity 
$$
(u,e,p) : [0,T] \to BV((0,L)) \times L^2((0,L)) \times \M( [0,L])
$$
satisfying \eqref{eq:def plasticity} and \eqref{eq:EB}, such that for all $t \in [0,T]$
$$
\begin{cases}
u_N(t) \rightharpoonup u(t) \text{ weakly-* in } BV((0,L)), \\
e_N(t) \rightharpoonup e(t) \text{ weakly in } L^2((0,L)), \\
p_N(t) \rightharpoonup p(t) \text{ weakly-* in } \M([0,L]),
\end{cases}
$$
when passing to the limit $\delta_N \rightarrow 0$. In particular, all the above quantities in \eqref{eq:incremental properties} pass to the limit as $\delta_N \to 0$. Therefore, Theorem \ref{thm:CNS FR} holds true for the quasi-static evolution $(u,e,p)$ as well. Consequently, for any prescribed boundary datum $w \in AC([0,T];H^1(\R))$ not satisfying \eqref{eq:CNS w}, we come to a contradiction. It unfortunately proves that commuting the limits in $\e$ and $N$ leads to no better statement.

\medskip
One could then argue that the One-sided Minimality of \cite[Theorem 2]{FG} might be too weak and that it might have been preferable to pass to the limit $\e \searrow 0$ in the time-continuous quasi-static evolution of \cite[Theorem 7]{GL} instead, since the Minimality condition considered in \cite[Definition 3]{GL} is stronger as stated in \cite[Remark 4]{GL}. Unfortunately again, as we are working in the one-dimensional setting, the two minimality conditions are equivalent. Indeed, let us fix $\e = 1$ for the present discussion. By optimality of the evolution \cite[Proposition 1]{FG}, there exists $\chi_n : [0,T] \to L^\infty((0,L);\{0,1\})$ non-decreasing in time such that $\chi_n(t) \rightharpoonup 1 - \Theta(t)$ weakly-* in $L^\infty((0,L))$ for all time $t \in [0,T]$. Setting $D_n := \{ \chi_n = 1 \}$, we have that for all $(A,\theta)$ such that $A \in \widehat \G_\theta ( \{ D_n (t)\}, a_0,a_1) = \G_\theta (a_0,a_1)$, $A \in G_{\tilde \theta}(a_0,a(t))$ with 
$$\tilde \theta := \frac{ \theta - (1 - \Theta(t))}{\Theta(t)} \mathds{1}_{\{ \Theta(t) > 0 \}} \in L^\infty((0,L);[0,1]).$$ Therefore, $A$ is an admissible competitor for the One-sided minimality of \cite[Theorem 2]{FG} too.

\medskip
Eventually, we conclude with a general remark by noticing that this one-dimensional analysis seems to raise the question whether Hencky perfect plasticity is distinguishable from damage or not in a static setting, as mentionned in the recent survey \cite[Section 1, p10]{M}.

\section{Appendix}\label{section:appendix}

\begin{lem}\label{lem:Convex Envelope}
Let $a,b,K >0$ with $a<b$ and $f : \xi \in \R \mapsto \min \left\{ K +  a |\xi|^2; b |\xi|^2 \right\}$. Then, for all $\xi \in \R$,
$$
\mathcal{C} f(\xi) =  \begin{cases} \vspace{0.2cm}
								 \,  b |\xi|^2 &  \text{if } |\xi| \leq \sqrt{\frac{a K}{b(b-a)}}, \\ \vspace{0.2cm}
							    \, |\xi| \sqrt{\frac{4ab K  }{b - a }} - \frac{ aK }{ b - a}  &  \text{if } \sqrt{\frac{a K}{b(b-a)}} < |\xi| \leq \frac{b}{a} \sqrt{\frac{a K}{b(b-a)}}, \\
								\, K  +  a |\xi|^2 &  \text{if } |\xi| > \frac{b}{a} \sqrt{\frac{a K}{b(b-a)}}.
								\end{cases} 
$$
\end{lem}

\begin{figure}[hbtp]
\begin{tikzpicture}[line cap=round,line join=round,>=triangle 45,x=0.85cm,y=0.85cm]
\clip(-7,-1.2) rectangle (7,6.5);
\draw[line width=1.pt,smooth,samples=100,domain=-6:6] plot(\x,{(0.1*(\x)^(2)+2.0)});
\draw[line width=1.pt,smooth,samples=100,domain=-2.5:2.5] plot(\x,{1.0*(\x)^(2)});
\draw[line width=1pt,color=ffqqqq,smooth,samples=100,domain=-6.5:6.5] plot(\x,{(abs((\x))*sqrt(4*1.0*0.1*2.0/(1.0-0.1))-0.1*2.0/(1.0-0.1))});
\draw (2.3,5.8) node[anchor=north west] {$ y = b\left\lvert \xi \right\rvert^2 $};
\draw (-0.1,2.5) node[anchor=north west] {$\kappa$};
\draw (-6.1,1.5) node[anchor=north west,color=ffqqqq] {$y = \left\lvert \xi \right\rvert \sqrt{\frac{4ab\kappa}{b-a}} - \frac{a\kappa}{b-a}$};
\draw (-5.7,5.9) node[anchor=north west] {$y = a \left\lvert \xi \right\rvert^2 + \kappa$};
\draw (-0.1,-0.22) node[anchor=north west] {$\sqrt{ \frac{a\kappa}{b(b-a)}}$};
\draw (3.5,-0.22) node[anchor=north west] {$\frac{b}{a} \sqrt{ \frac{a\kappa}{b(b-a)}}$};
\draw [line width=1.pt] (-5.5,0.)-- (6.,0.);
\draw [line width=1.pt] (5.8,0.1)-- (6.,0.);
\draw [line width=1.pt] (5.8,-0.1)-- (6.,0.);
\draw [line width=1.pt] (0.,-1.)-- (0.,6.5);
\draw [line width=1.pt] (0.,6.5)-- (-0.1,6.2);
\draw [line width=1.pt] (0.,6.5)-- (0.1,6.2);
\draw (6.197429188010825,0.4143492978980675) node[anchor=north west] {$\xi$};
\begin{scriptsize}
\draw [fill=uuuuuu] (-4.714045166906035,4.222222183563016) circle (0.5pt);
\draw [fill=uuuuuu] (4.714045121532328,4.2222221407842735) circle (0.5pt);
\draw [fill=uuuuuu] (-0.4714045454596425,0.22222224548001215) circle (0.5pt);
\draw [fill=uuuuuu] (0.4714045447969268,0.22222224485519776) circle (0.5pt);
\draw [line width=0.7pt, dotted] (0.4714045454596425,0.22222224548001215)-- (0.4714045454596425,-0.2);
\draw [line width=0.7pt, dotted] (4.714045121532328,4.2222221407842735)-- (4.714045121532328,-0.2);
\end{scriptsize}
\end{tikzpicture}
\caption{}
\label{fig:convex envelop}
\end{figure}

\begin{prop} Let $\O \subset \R^N$ be an open bounded set with Lipschitz boundary and $A_0$ and $A_1$ be two symmetric fourth order isotropic tensors, {\it i.e.} 
$$A_i \xi = \lambda_i {\rm tr}(\xi) {\rm Id} + 2 \mu_i \xi $$
for all $\xi \in \Ms$ where $\lambda_1 > \lambda_0 > 0$ and $\mu_1 > \mu_0 > 0$ are the Lam\'e coefficients. For a fixed toughness $\kappa >0$, we consider the closed convex set 
$$\KK = \{ \tau \in \Ms \, : \, G(\tau) \leq 2 \kappa \}$$
with $G$ defined by \eqref{eq:G} in the Introduction. We define 
$$\bar W (\xi) = \left( \frac12 A_1 {\rm id}:{\rm id} \right) \infconv I_\KK^* (\xi) $$
and 
$$W_\e ( \xi) = \min \, \left\{ \frac{\kappa}{\e} + \frac12 \e A_0 \xi:\xi ; \frac12 A_1 \xi:\xi \right\}$$
for all $\e > 0$ and all $\xi \in \Ms$. Then, given a boundary datum $w \in H^1(\R^N;\R^N)$, the functional $ \F_\e : L^1(\O;\R^N) \to \bar{ \R^+}$ defined for all $u \in L^1(\O;\R^N)$ by
$$
 \F_\e(u) := \begin{cases}
{\ds \int_\O  SQ W_\e(e(u)) \, dx } & \quad \text{ if } u \in w + H^1_0(\O;\R^N) \\
+ \infty & \quad \text{ otherwise}
\end{cases}
$$
$\Gamma$-converges in $L^1(\O;\R^N)$ as $\e \searrow 0$ to the functional
$$
 \F(u) := \begin{cases}
 \begin{array}{r}
 {\ds \int_\O \bar W(e(u)) \, dx + \int_\O  I_\KK^* \left( \frac{ d E^su}{d|E^su|} \right) \, d |E^su| + \int_{\partial \O} I_\KK^* \left( (w-u) \odot \nu \right) \, d \HH^{N-1} } \\
 \quad \text{ if } u \in BD(\O)
\end{array} \\
+ \infty \quad \text{ otherwise.}
\end{cases}
$$
\end{prop}

\begin{proof} 
The proof of the $\Gamma$-lim inf inequality presents no particular difficulty and we do not give the details of its proof. The key point is to extend approximating sequences $u_k \to u$ by $w$ in a larger open bounded set $\O \subset \subset \O'$ and rely on \cite[Theorem 3.1]{BIR} in $\O'$ to conclude. 

\medskip
The proof of the $\Gamma$-lim sup inequality relies on the approximation result \cite[Theorem 3.5]{Mora} in $BD$, which allows us to reduce to the case where $u \in w + C^\infty_c(\O;\R^N)$. Indeed, let $u \in BD(\O)$. Using the Radon-Nikod\'ym decomposition of $E u$ with respect to Lebesgue and owing to a measurable selection criterion together with the definition of the inf-convolution, there exist $e$ and $p^a$ in $ L^2(\O;\Ms)$ such that 
\begin{gather*}
Eu = e(u) \LL^N \res \O + E^s u, \quad e(u) = e + p^a \\
\text{and }\quad  \bar W (e(u)) = \frac12 A_1 e:e + I_\KK^*(p^a) \quad \LL^N \text{-a.e. in } \O.
\end{gather*}
We define $p = E^s u \mathds{1}_{\O} + p^a \LL^N \res \O + (w-u) \odot \nu \HH^{N-1} \res \partial \O \in \M(\bar \O;\Ms)$. Following the proof of \cite[Theorem 3.5]{Mora} with minor adaptations, one can find a sequence 
$$(u_k,e_k,p_k) \in LD(\O) \times L^2(\O;\Ms) \times L^1(\O;\Ms)$$
such that 
\begin{gather*}
e(u_k) = e_k + p_k \quad \text{for all } k \in \N ,\\
(u_k - w, e_k - e(w), p_k) \in C^\infty_c (\O;\R^N) \times C^\infty_c (\O;\Ms) \times C^\infty_c (\O;\Ms) 
\end{gather*}
and 
\begin{gather*}
u_k \to u \quad \text{strongly in } L^1(\O;\R^N), \quad e_k \to e \quad \text{strongly in } L^2(\O;\Ms), \\
p_k \rightharpoonup p \quad \text{weakly-* in } \M \big( \bar \O;\Ms \big) \quad \text{and} \quad \int_\O \left\lvert p_k \right\rvert \, dx \to \left\lvert p \right\rvert \big( \bar \O \big)
\end{gather*}
when $k \nearrow \infty$. By Reshetnyak's Continuity Theorem (see \cite[Theorem 2.39]{AFP}), we get that
\begin{multline*}
\int_\O I_\KK^*(p_k) \, dx \to \int_{\bar \O} I_\KK^* \left( \frac{ d p}{d \left\lvert p \right\rvert } \right) \, d \left\lvert p \right\rvert \\
= \int_\O I_\KK^*(p^a) \, dx + \int_\O I_\KK^* \left( \frac{ d E^s u}{d \left\lvert E^s u \right\rvert} \right) \, d \left\lvert E^s u \right\rvert + \int_{\partial \O} I_\KK^* \left( (w-u) \odot \nu \right) \, d \HH^{N-1}
\end{multline*}
when $k \nearrow \infty$. Therefore,  by definition of the inf-convolution and by lower semi-continuity of the $\Gamma$-upper limit, denoted by $\F''$, if $\F'' (u_k) \leq \F(u_k)$ we would get that
\begin{multline*}
 \limsup_{ k \nearrow \infty} \int_\O \bar W \big(e(u_k) \big) \, dx \\
 \leq \int_\O \frac12 A_1 e:e \, dx + \int_\O I_\KK^*(p^a) \, dx + \int_\O I_\KK^* \left( \frac{ d E^s u}{d \left\lvert E^s u \right\rvert} \right) \, d \left\lvert E^s u \right\rvert  \\
 + \int_{\partial \O} I_\KK^* \left( (w-u) \odot \nu \right) \, d \HH^{N-1} = \F(u).
\end{multline*}
Therefore, we can assume without loss of generality that $u \in w + C^\infty_c(\O;\R^N)$ and conclude the proof of the upper bound arguing as in \cite[Proposition 3.3]{BIR}.
\end{proof}

\end{document}